\documentclass[a4paper]{amsart}
\usepackage{amsmath,amsthm,amssymb,latexsym,epic,bbm,comment,mathrsfs}
\usepackage{graphicx,enumerate,stmaryrd,xcolor}
\usepackage[all,2cell]{xy}
\xyoption{2cell}

\newtheorem{theorem}{Theorem}
\newtheorem{lemma}[theorem]{Lemma}
\newtheorem{corollary}[theorem]{Corollary}
\newtheorem{proposition}[theorem]{Proposition}

\newtheorem{example}[theorem]{Example}

\newtheorem{remark}[theorem]{Remark}
\usepackage[all]{xy}
\usepackage[active]{srcltx}
\usepackage[parfill]{parskip}
\usepackage{enumerate}

\begin{document}
\title[Simple gradable $\mathbb{C}\lbrack\mathtt{h}\rbrack$-torsion 
free $\mathfrak{sl}_2$-modules]{Simple gradable
$\mathbb{C}\lbrack\mathtt{h}\rbrack$-torsion 
free $\mathfrak{sl}_2$-modules}

\author[V.~Mazorchuk and S.~Xu]{Volodymyr Mazorchuk and Shuoyang Xu}

\begin{abstract}
We give an explicit classification as well as an explicit 
construction of simple $\mathfrak{sl}_2$-modules that are torsion-free
of finite rank with respect to the action of the Cartan subalgebra
and which admit a grading by a finite cyclic group corresponding 
to the rank of the module.
\end{abstract}

\maketitle

\section{Introduction and description of the results}\label{s0}

\subsection{Background}\label{s0.1}

Classification of irreducible representations of a given algebraic
structure is a fundamental problem in contemporary representation
theory. In full generality, this problem is usually very difficult.
But it often admits reasonable answers under some restrictions.

For simple complex Lie algebras, there are many interesting 
classifications of various classes of simple modules, including
finite dimensional modules, highest weight modules, 
weight modules, Whittaker modules, Gelfand-Zeitlin modules and 
others, see e.~g. \cite{Hu,Mat,KTWWY,Ni,Ni2} and references therein.
The only  simple complex Lie algebra for which 
some kind of general answer is known is the case of the Lie algebra
$\mathfrak{sl}_2$, see \cite{Bl79,Bl81}. Unfortunately, this answer
is not explicit but rather reduces the problem to classification of 
the similarity classes of left irreducible elements over a certain
non-commutative principal ideal domain. 

Recently, there were a few studies which attempted to make the answer
of \cite{Bl79,Bl81} explicit in some special cases. For example, 
the main result of \cite{GKM} gives an explicit classification of 
simple $\mathfrak{sl}_2$-modules that are torsion-free of rank one over
the Cartan subalgebra. The main result of \cite{GNZ} provides an explicit
construction and classification of a family of 
simple $\mathfrak{sl}_2$-modules that are free of rank two over
the Cartan subalgebra. Interesting families of 
simple $\mathfrak{sl}_2$-modules that are free of finite 
(but arbitrarily large) rank over
the Cartan subalgebra were constructed in \cite{LN}.

\subsection{Results}\label{s0.2}

In this paper we generalize, in some sense, the approach and the
results of \cite{GKM}.  The algebra $\mathfrak{sl}_2$ has a natural
$\mathbb{Z}$-grading in which the Cartan part has degree $0$,
the positive root generator has degree $1$ and the negative root generator
has degree $-1$. For any positive integer $m$, we obtain the induced
grading by the cyclic group $\mathbf{C}_m\cong \mathbb{Z}/m\mathbb{Z}$.
The main aim of the present paper is to provide
an explicit classification and construction of simple 
$\mathfrak{sl}_2$-modules that are torsion-free of rank $m$ over the Cartan 
subalgebra and admit a $\mathbf{C}_m$-grading. The results of 
\cite{GKM} can be interpreted as corresponding to the (degenerate, 
in some sense) case $m=1$.

Our classification result is an elegant generalization of the main result of 
\cite{GKM}. Let $m$ be a positive integer. 
After fixing a central character $\vartheta$ of
$U(\mathfrak{sl}_2)$, we choose a real number $\omega$
in a certain way (depending on $\vartheta$) 
and consider the set $\Omega_\omega$ of all
complex numbers that have the real part inside the half-interval 
$[\omega,\omega +2)$. Now we consider the set $\Psi^{(m)}$ that 
consists of all tuples $(c,\zeta_0,\dots,\zeta_{m-1})$,
where $c\in \mathbb{C}^\times$ and all $\zeta_i$ are functions
from $\Omega_\omega$ to $\mathbb{Z}$ with finite support.
The group $\mathbf{C}_m$ acts on $\Psi^{(m)}$ 
by cyclic permutation of the $\zeta_i$'s. Our main classification 
result is the following statement, see Theorem~\ref{thm-s1.8-9}:

{\bf Theorem A.} {\em Let $m$ be a positive integer.
Then the isomorphism classes of simple 
$\mathfrak{sl}_2$-modules that are torsion-free of rank $m$ over the Cartan 
subalgebra and admit a $\mathbf{C}_m$-grading
are in bijection with the regular orbits of $\mathbf{C}_m$
on $\Psi^{(m)}$. }

Each tuple $\boldsymbol{\zeta}:=(c,\zeta_0,\dots,\zeta_{m-1})$ corresponds to 
a non-zero element $g_{\boldsymbol{\zeta}}$ in the field $\mathbb{C}(\mathtt{h})$  
of rational functions in the element $\mathtt{h}$
(the usual generator of the Cartan subalgebra). 
We have the automorphism of
$\sigma$ of $\mathbb{C}(\mathtt{h})$ which sends $\mathtt{h}$
to $\mathtt{h}-2$ and hence we can consider the corresponding
skew Laurent polynomial algebra $R:=\mathbb{C}(\mathtt{h})[x,x^{-1},\sigma]$.
There is an embedding of the primitive quotient of 
$U(\mathfrak{sl}_2)$ corresponding to $\vartheta$ into
$R$ which allows one to consider $R$-modules as
$\mathfrak{sl}_2$-modules. The simple 
$\mathfrak{sl}_2$-module corresponding to $\boldsymbol{\zeta}$ turns
out to be the socle of $R/R(x^m-g_{\boldsymbol{\zeta}})$.
Our explicit description of this socle, see Theorem~\ref{thm-s2.5-7}
for details, can be summarized as follows:

{\bf Theorem B.} {\em Given a tuple $\boldsymbol{\zeta}$
as above, the corresponding simple $\mathfrak{sl}_2$-module is 
the $\mathfrak{sl}_2$-submodule of 
$R/R(x^m-g_{\boldsymbol{\zeta}})$
which is described by an explicit $\mathbb{C}$-basis.
}

As a bonus, we also obtain similar results for the first Weyl algebra.
As an interesting application, we obtain a very neat
irreducibility criterion for elements of $R$ that have the form
$x^m-\beta\in R$, where $\beta\in \mathbb{C}(\mathtt{h})$,
see Corollary~\ref{cor-lirr}, significantly improving
and generalizing
the earlier results in the literature, see, for example, 
\cite[Corollary~2]{BP}.

\subsection{Methods}\label{s0.3}

Our main idea is based on the results of \cite{Bl79,Bl81},
namely, on the connection between $\mathfrak{sl}_2$-modules
and $R$-modules mentioned above. In more detail, there is a
bijection between simple $R$-modules and simple 
$\mathfrak{sl}_2$-modules with a fixed central character, 
the action of the Cartan subalgebra  on which is torsion-free. 
The bijection is given by restricting to the image of the 
embedding and then taking the socle. This brings us to the
following two challenges:
\begin{itemize}
\item Give an explicit classification of simple $R$-modules
that have dimension $m$ over $\mathbb{C}(\mathtt{h})$ and admit a $\mathbf{C}_m$-grading.
\item For each simple $R$-module obtained in the previous item,
give an explicit description of its $\mathfrak{sl}_2$-socle.
\end{itemize}

The algebra $R$ is a non-commutative principal ideal domain.
Therefore its simple modules of dimension $m$ over $\mathbb{C}(\mathtt{h})$
correspond to left irreducible elements of the form
\begin{displaymath}
x^m-\alpha_1 x^{m-1}-\dots -\alpha_m,
\text{ where } \alpha_i\in \mathbb{C}(\mathtt{h}).
\end{displaymath}
The condition of existence of a $\mathbf{C}_m$-grading restricts this
further to the elements of the form $x^m-\beta$, for $\beta\in \mathbb{C}(\mathtt{h})$.
So, our Theorem~A reduces to classification of the similarity classes of
left irreducible elements of such form. To obtain this,
for small $m$, we mostly used some inspired computations and combinatorics 
as well as significant  help of Microsoft Copilot, from which we learned 
a lot about different approaches and methods of doing algebra in
difference ring settings, see e.g. Theorem~\ref{thm-s1.5-3}. 
However, the case of an arbitrary $m$ requires a more sophisticated
approach rooted in representation theory.

Given Theorem~A, our Theorem~B reduces to a careful study of the restriction of
a certain explicit simple $R$-module to $\mathfrak{sl}_2$.
Here the proof contains two major steps: at the first step 
we have to determine  one non-zero element in the socle of such restriction.
At the second step, we have to determine the whole 
$\mathfrak{sl}_2$-submodule which this non-zero element generates. 
Both steps reduce, to a very large extent, to a careful and meticulous
combinatorial case-by-case walk-through.

\subsection{Structure}\label{s0.4}

The paper is split naturally into three parts.  
In Section~\ref{s1} we study various types of $R$-modules.
Due to the connection between $R$-modules and 
$\mathfrak{sl}_2$-modules described in \cite{Bl79,Bl81},
the problem of explicit classification of simple $R$-modules
is interesting. This problem seems to be hard, in general.
In Section~\ref{s1} we investigate various special cases, obtain
some partial results and point out a few challenges.
Our main result in this section is Theorem~\ref{thm-s1.8-9}.

In Section~\ref{s2} we apply the results of 
Section~\ref{s1} to obtain our principal applications to 
the study of  $\mathfrak{sl}_2$-modules.
In Section~\ref{s3}, which can be viewed as a bonus,
we apply the results of Section~\ref{s1} to
obtain results, similar to those from Section~\ref{s2},
in the setup of the first Weyl algebra.

\subsection*{Acknowledgements}
The first author is partially supported by the Swedish Research Council.
The second author is partially supported by the Fundamental 
Research Funds for the Central Universities
(No.~B250201222), and by the Young Faculty Overseas 
Training Program, which is jointly funded by Hohai
University and the China Scholarship Council.
The research was done during the visit of the second author
to Uppsala University whose  hospitality is gratefully acknowledged.
Several ideas that are discussed in 
Section~\ref{s1} are inspired by conversations with 
Microsoft Copilot and ChatGPT+.
\vspace{5mm}
 
\section{The algebra $R$ and $R$-modules}\label{s1}

\subsection{The skew Laurent polynomial algebra $R$}\label{s1.1}

Consider the field $\Bbbk:=\mathbb{C}(\mathtt{h})$ of rational 
functions in the variable $\mathtt{h}$ with complex coefficients. Let
$\sigma$ be the automorphism of $\Bbbk$ sending $\mathtt{h}$
to $\mathtt{h}-2$. To this datum, we can associate the 
{\em skew Laurent polynomial algebra} $R=\Bbbk[x,x^{-1},\sigma]$
defined as a free $\Bbbk$-module with basis $x^i$, where $i\in\mathbb{Z}$,
and in which the multiplication is defined via
\begin{displaymath}
x^i\cdot x^i:=x^{i+j}, \text{ for }i,j\in\mathbb{Z}, 
\end{displaymath}
as well as
\begin{displaymath}
xa=\sigma(a)x,\text{ for }a\in \Bbbk.
\end{displaymath}
The field $\Bbbk$ is an example of a {\em difference field}
in the terminology of \cite{PS}. We note that
the subfield $\Bbbk^\sigma$ of $\sigma$-invariants in $\Bbbk$
coincides with $\mathbb{C}$. This will be important for
some arguments later.

We note that $R$ is an Euclidean ring in the obvious way, 
in particular, it is a non-commutative principal ideal 
domain (both left and right).

We denote by $\mathbf{G}$ the subgroup of $\Bbbk^\times$
consisting of all elements of the form $\frac{\sigma(a)}{a}$,
where $a\in\Bbbk^\times$. Note that $\mathbf{G}$ is the 
group of $1$-coboundaries for the action of the cyclic 
group generated by $\sigma$ on $\Bbbk^\times$.

We recall that the algebra $\mathbb{C}(\mathtt{h})$  has a 
$\mathbb{C}$-basis, which we will denote by  $\mathbf{B}$,
consisting of all monomials  $\mathtt{h}^i$, 
where $i\geq 0$, as well as all fractions 
$\frac{1}{(\mathtt{h}-t)^m}$, where $t\in \mathbb{C}$
and $m\in\mathbb{Z}_{>0}$.

\subsection{Left ideals}\label{s1.15}

Being a principal ideal domain, each left ideal $I$ in $R$
is principal, that is, there exists $a\in R$ such that 
$I=Ra$. Each $a\in R$ can be written uniquely as
\begin{displaymath}
a=\sum_{i\in \mathbb{Z}}a_i x^i, 
\end{displaymath}
where $a_i\in \Bbbk$ and only finitely many of the $a_i$'s
are non-zero. If $a\neq 0$, then there is $j\in\mathbb{Z}$
such that $a_j\neq 0$ and $a_i=0$, for all $i>j$. In this
case $j$ is called the {\em height} of $a$ and the corresponding
$a_j$ is called the {\em leading term}. A non-zero $a\in R$
is called {\em monic} provided that its leading term equals $1$.
The {\em degree} of $a$ is the difference between the height 
and the minimal $l$ such that $a_l\neq 0$.

We will call a non-zero $a\in R$ {\em reduced} provided that 
it is monic and, additionally, satisfies $a_0\neq 0$
and $a_i=0$, for $i<0$. 

\begin{lemma}\label{lem0}
The map $a\mapsto Ra$ is a bijection between the set of all
reduced elements of $R$ and the set of all non-zero left ideals.
\end{lemma}

\begin{proof}
The map is obviously well-defined.
Let $I$ be a non-zero ideal of $R$ and $b\in R$
be such that $I=Rb$. Then $b\neq 0$. Multiplying $b$ on the left with
the invertible elements $x^{\mp 1}$, which, clearly, 
does not affect $I$, we can substitute $b$ by some $b'$
which satisfies $b'_0\neq 0$ and $b'_i=0$, for $i<0$.
Now, dividing by the (invertible) leading term of $b$,
we obtain a reduced element that generates $I$.
Hence our map is surjective.

To prove injectivity, assume $Ra=Rb$ with both
$a$ and $b$ reduced. Then $R/Ra=R/Rb$ and hence
$a$ and $b$ have the same degree. If $a\neq b$, 
then $a-b\in Ra$ is non-zero and has strictly 
smaller degree. Therefore the $\Bbbk$-dimension
of $R/R(a-b)$ is strictly smaller than that 
of $R/Ra$, contradicting $R(a-b)\subset Ra$.
\end{proof}

\begin{lemma}\label{lem01}
Let $a\in R$ be such that $0\subsetneq Ra\subsetneq R$.
If $R/Ra$ is a simple $R$-module and $k=\dim_\Bbbk(R/Ra)$, 
then, for any  $b\not\in Ra$, the elements 
\begin{equation}\label{eqeq2}
b+Ra,\,xb+Ra,\,x^2b+Ra,\,\dots,\, x^{k-1}b+Ra  
\end{equation}
form a basis of $R/Ra$ and the matrix of the action of 
$x$ in this basis is given by
\begin{displaymath}
\left(
\begin{array}{cccccccc}
0&0&0&0&0&\dots&0&\alpha_0\\ 
1&0&0&0&0&\dots&0&\alpha_1\\ 
0&1&0&0&0&\dots&0&\alpha_2\\ 
0&0&1&0&0&\dots&0&\alpha_3\\ 
0&0&0&1&0&\dots&0&\alpha_4\\ 
\vdots&\vdots&\vdots&\vdots&\vdots&\ddots&\vdots&\vdots\\ 
0&0&0&0&0&\dots&0&\alpha_{k-2}\\ 
0&0&0&0&0&\dots&1&\alpha_{k-1}\\ 
\end{array}
\right),
\end{displaymath}
for some $\alpha_0,\alpha_1,\dots,\alpha_{k-1}\in \Bbbk$.
Moreover, we have 
\begin{displaymath}
\mathrm{Ann}_R(b+Ra)=R(x^k-\alpha_{k-1}x^{k-1}-\alpha_{k-2}x^{k-2}-\dots-\alpha_{1}x-\alpha_0). 
\end{displaymath}
\end{lemma}

\begin{proof}
If the elements in Equation~\eqref{eqeq2} are linearly independent over $\Bbbk$,
they form a $\Bbbk$-basis (since their number equals the dimension)
and then the form of the matrix describing the action of $x$ 
in this basis follows from the
definitions. From this matrix, we have that $\mathrm{Ann}_R(b+Ra)$ 
contains $R(x^k-\alpha_{k-1}x^{k-1}-\alpha_{k-2}x^{k-2}-\dots-\alpha_{1}x-\alpha_0)$.
If this containment were proper, then $R/\mathrm{Ann}_R(b+Ra)$ would have 
strictly smaller  $\Bbbk$-dimension than $k$, contradicting simplicity
of $R/Ra$.

If the elements in Equation~\eqref{eqeq2} were linearly dependent over $\Bbbk$,
their $\Bbbk$-linear span would be a proper non-zero submodule of $R/Ra$,  
contradicting our assumption that $R/Ra$ is simple.
\end{proof}

Let $a\in R$ be non-zero and not invertible. 
Then $a$ is called {\em left irreducible} provided that
$Ra$ is a maximal left ideal, equivalently, the 
$R$-module $R/Ra$ is simple. Two left irreducible elements
$a,b\in R$ are called {\em similar}, denoted $a\sim b$,
provided that $R/Ra\cong R/Rb$. Thus, by definition, 
the isomorphism classes of simple $R$-modules are
in bijection with the similarity classes of left 
irreducible elements in $R$. Note that 
$a\sim b$ is equivalent to the following condition:
there is $v\in R\setminus Ra$ such that 
$\mathrm{Ann}_R(v+Ra)=Rb$. This is, in turn, equivalent
to the following condition: there exist 
$v\in R\setminus Ra$ and $u\in R$ such that $bv=ua$.

\subsection{$\Bbbk$ as an $R$-module}\label{s1.2}

An $R$-module structure on $\Bbbk$ is uniquely determined
by a choice of $u\in \Bbbk$, interpreted as the image of
$1$ under $x$, that is: $x\cdot 1=u$. Then, for any $a\in\Bbbk$
we have $x\cdot a=\sigma(a)u$. We denote this 
$R$-module by $\Bbbk_u$. It is easy to check  that 
$\Bbbk_u\cong \Bbbk_v$ if and only if $\frac{u}{v}\in\mathbf{G}$,
see \cite[Section~2]{GKM} for details. We also note that each
$\Bbbk_u$ is a simple $R$-module.

\subsection{$R$-modules that are finite dimensional over $\Bbbk$}\label{s1.3}

Let $M$ be an $R$-module that is finite dimensional over $\Bbbk$.
The $R$-module structure on $M$ is uniquely determined by 
a choice of an invertible $\sigma$-skew linear operator $F=F_M$ on $M$
which describes the action of $x$. If we fix a $\Bbbk$-basis
$\mathbf{m}:=(m_1,\dots,m_r)$ in $M$, we can represent 
$F$ via its matrix $[F]=[F]_{\mathbf{m}}^{\mathbf{m}}=(f_{i,j})$,
where $i,j\in\{1,2,\dots,r\}$,
defined, as usual, via
\begin{displaymath}
F(m_j)=\sum_{i=1}^r  f_{i,j} m_i.
\end{displaymath}

Given two $R$-modules $M$ and $N$, an $R$-homomorphism
$\varphi$ from $M$ to $N$ is a $\Bbbk$-linear map which 
intertwines $F_M$ and $F_N$, that is: $F_N\circ \varphi=\varphi\circ F_M$.
Given a basis $\mathbf{m}$ in $M$ and a basis $\mathbf{n}$ in $N$,
we can represent $\varphi$, as usual, by its matrix 
$[\varphi]=[\varphi]_{\mathbf{n}}^{\mathbf{m}}$. The fact that 
$F$ is $\sigma$-skew linear means that the intertwining condition,
written in the matrix form, becomes
\begin{displaymath}
[F_N]_{\mathbf{n}}^{\mathbf{n}}\,
\sigma([\varphi])_{\mathbf{n}}^{\mathbf{m}}=
[\varphi]_{\mathbf{n}}^{\mathbf{m}}\,
[F_M]_{\mathbf{m}}^{\mathbf{m}}.
\end{displaymath}
Here is a general analogue of Schur's lemma in our setup
(several ideas in this proof were suggested by Microsoft Copilot
with a general reference to \cite{PS}).

\begin{lemma}\label{lem1}
Assume that $M$ as above is simple. Then
$\mathrm{End}_{R}(M)\cong \mathbb{C}$.
\end{lemma}

We note that $R$ has uncountable dimension over $\mathbb{C}$
and hence the usual argument (due to Dixmier) for establishing
the claim of Lemma~\ref{lem1}, 
see e.g. \cite[Proposition~2.6.8]{Di}, is not directly applicable. 

\begin{proof}
From the classical Schur's lemma, we know that 
$\mathrm{End}_{R}(M)$  is a division algebra
over $\mathbb{C}$. So, to complete the proof, we just need to
show that $\mathrm{End}_{R}(M)$ is finite dimensional over 
$\mathbb{C}$. Let $r=\dim_\Bbbk(M)$. Our aim is to show that
any collection of $s>r^2$ elements in $\mathrm{End}_{R}(M)$
are linearly dependent over $\mathbb{C}$.

So, let $v_1,\dots,v_s\in \mathrm{End}_{R}(M)$ be such that
$s>r^2$. Then these elements are linearly dependent over
$\Bbbk$ since $s>\dim_\Bbbk(\mathrm{End}_{\Bbbk}(M))$. We choose the standard
basis in $\mathrm{End}_{\Bbbk}(M)$ and let $Q$ be the matrix
obtained by writing the coordinates of the elements 
$v_1,\dots,v_s$ with respect to this basis as columns. Then there
is $0\neq u\in \Bbbk^s$ such that $Qu=0$. Assume that 
the number of the non-zero coordinates in $u$ is minimal possible.
Let $u_i\neq 0$ be a non-zero coordinate of $u$. Then we
can rescale $u$ such that $u_i=1$.

Consider the invertible $\Bbbk$-linear operator $\Psi$ on
$\mathrm{End}_{\Bbbk}(M)$ given by 
$U\mapsto [F_M]^{-1}U[F_M]$ and let $\llbracket F_M\rrbracket$
be its matrix in the basis chosen above.
By construction, each $v_j$ in $Q$ satisfies
$\sigma(v_j)=\Psi(v_j)$ and hence we have
$\sigma(Q)= \llbracket F_M\rrbracket Q$. 
Starting from $Qu=0$, we have $\llbracket F_M\rrbracket Qu=0$.
Using $\llbracket F_M\rrbracket Q= \sigma(Q)$, we thus get
$\sigma(Q)u=0$ and therefore we deduce $Q\sigma^{-1}(u)=0$
by applying $\sigma^{-1}$.

Note that $\sigma^{-1}(u_i)=\sigma^{-1}(1)=1=u_i$.
From $Qu=Q\sigma^{-1}(u)=0$, we deduce that $Q(u-\sigma^{-1}(u))=0$.
However, all zero coordinates of $u$ remain zero
in $u-\sigma^{-1}(u)$ and, additionally,
the $i$-th coordinate of $u-\sigma^{-1}(u)$ is zero.
Due to our minimality assumption, this forces 
$u-\sigma^{-1}(u)=0$, that is, $u$ is $\sigma$-invariant.
Consequently, all coordinates of $u$ belong to
$\mathbb{C}$ and we are done.
\end{proof}

\subsection{Simple $R$-modules that are two dimensional over $\Bbbk$}\label{s1.4}

For $\alpha,\beta\in \Bbbk$, with $\beta\neq 0$, consider the reduced 
element $x^2-\alpha x -\beta$ and the corresponding $R$-module
\begin{displaymath}
M_{\alpha,\beta}:=R/R(x^2-\alpha x -\beta). 
\end{displaymath}
The elements $1$ and $x$ give rise to an obvious $\Bbbk$-basis 
of the module 
$M_{\alpha,\beta}$, in particular, $\dim_\Bbbk(M_{\alpha,\beta})=2$.
The matrix of $F$ in this basis has the
following form, see Lemma~\ref{lem01}:
\begin{displaymath}
[F]=\left(\begin{array}{cc}0&\beta\\1&\alpha\end{array}\right). 
\end{displaymath}
The module $M_{\alpha,\beta}$ is simple if and only if 
$x^2-\alpha x -\beta$ is left irreducible.
The following lemma is taken from \cite[Theorem~1]{BP},
we include a short proof for convenience of the reader:

\begin{lemma}\label{lem-s1.4-1}
For $\alpha,\beta\in \Bbbk$, with $\beta\neq 0$, the element
$x^2-\alpha x -\beta$ is left irreducible if and only if the
following equation (with unknown $y$) does not have any 
solutions in $\Bbbk$:
\begin{equation}\label{eq-irreq}
y\sigma(y)-\alpha y -\beta =0. 
\end{equation}
\end{lemma}

\begin{proof}
The element $x^2-\alpha x -\beta$ is left irreducible if and only if
there do not exist $p,q\in\Bbbk$ such that 
\begin{displaymath}
x^2-\alpha x -\beta=(x-p)(x-q)=x^2-(p+\sigma(q))x+pq.
\end{displaymath}
Equivalently, we get the system
\begin{displaymath}
\left\{
\begin{array}{rcl}
\alpha &=&  p+\sigma(q)\\
\beta   &=& -pq.
\end{array}
\right.
\end{displaymath}
From the first equation, we can write $p=\alpha-\sigma(q)$.
Substituting into the second equation
gives $q\sigma(q)-\alpha q-\beta=0$. The claim of the lemma follows.
\end{proof}

\begin{remark}\label{rem-s1.2-5}
{\em
For
$\gamma\in\Bbbk^\times$, sending $x$ to $\gamma x$
gives rise to an automorphism of $R$.
Consequently, $x^2-\alpha x -\beta$ is left irreducible if and only if
$x^2-\frac{\alpha}{\sigma(\gamma)} x -\frac{\beta}{\sigma(\gamma)\gamma}$ 
is left irreducible. In particular, if $\alpha\neq 0$, then,  
by choosing $\gamma=\sigma^{-1}(\alpha)$, we reduce the question of the left 
irreducibility of $x^2-\alpha x -\beta$ to that for the element
$x^2-x -\frac{\beta}{\alpha\sigma^{-1}(\alpha)}$.
}
\end{remark}

\begin{proposition}\label{prop-s1.4-2}
Assume that $x^2-\alpha x -\beta$ is left irreducible 
and let $(p,q)\in \Bbbk^2\setminus \{(0,0)\}$.
Then, for $v:=p+qx$, the elements $v$ and $x(v)$ form
a $\Bbbk$-basis of $M_{\alpha,\beta}$ which we denote by
$\mathbf{v}$. Moreover, in this new basis, we have
\begin{displaymath}
[F]_\mathbf{v}^\mathbf{v}=
\left(
\begin{array}{cc}
0&\frac{\sigma(\Delta)}{\Delta}\beta\\
1&\frac{p\,\sigma^2(q)\,\sigma(\beta)\;+\;(p\alpha-q\beta)
\bigl(\sigma^2(p)+\sigma^2(q)\sigma(\alpha)\bigr)}{\Delta}
\end{array}
\right),
\end{displaymath}
where $\Delta=p(\sigma(p)+\sigma(q)\alpha)-
q\sigma(q)\beta\neq 0$.
\end{proposition}

\begin{proof}
Applying $x$ to $v:=p+qx$ gives
\begin{equation}\label{eq-p1}
\sigma(p)x+\sigma(q)\left(\beta+\alpha x\right)=
\sigma(q)\beta + (\sigma(p)+\sigma(q)\alpha) x.
\end{equation}
This is the second element of our basis. Note that the
fact that $x$ and $x(v)$ form a basis, which we have for free
since $M_{\alpha,\beta}$ is simple and $(p,q)\neq (0,0)$,
is equivalent to the condition that
\begin{displaymath}
\Delta:=p(\sigma(p)+\sigma(q)\alpha)-
q\sigma(q)\beta\neq 0.
\end{displaymath}
The fact that $\Delta\neq 0$ also follows from
Remark~\ref{rem-s1.2-5} (applied to $\gamma=-1$) combined with 
Lemma~\ref{lem-s1.4-1} (applied to $z=\frac{p}{q}$).

Applying $x$ to the element in Equation~\eqref{eq-p1}, outputs
\begin{displaymath}
\sigma^2(q)\sigma(\beta)x + (\sigma^2(p)+\sigma^2(q)\sigma(\alpha))
\left(\beta+\alpha x\right)
\end{displaymath}
which simplifies to 
\begin{displaymath}
(\sigma^2(p)+\sigma^2(q)\sigma(\alpha))\beta+
\left( \sigma^2(q)\sigma(\beta) +
(\sigma^2(p)+\sigma^2(q)\sigma(\alpha))\alpha \right)x.
\end{displaymath}
Now we need to express this latter element in the basis
of $v$ and $x(v)$. This reduces to solving the nonhomogeneous
system of linear equations with the following matrix:
\begin{displaymath}
\left(
\begin{array}{cc|c}
p&\sigma(q)\beta&(\sigma^2(p)+\sigma^2(q)\sigma(\alpha))\beta\\
q&\sigma(p)+\sigma(q)\alpha&\sigma^2(q)\sigma(\beta) +
(\sigma^2(p)+\sigma^2(q)\sigma(\alpha))\alpha
\end{array}
\right) 
\end{displaymath}
Using Cramer's Rule, the coefficient at $v$ is
\begin{displaymath}
\frac{(\sigma^2(p)+\sigma^2(q)\sigma(\alpha))\beta(\sigma(p)+\sigma(q)\alpha)
-(\sigma^2(q)\sigma(\beta) +
(\sigma^2(p)+\sigma^2(q)\sigma(\alpha))\alpha)\sigma(q)\beta}{\Delta} 
\end{displaymath}
and this simplifies to $\frac{\sigma(\Delta)}{\Delta}\beta$.
Similarly, the coefficient at $x(v)$ is
\begin{displaymath}
\frac{p(\sigma^2(q)\sigma(\beta) +
(\sigma^2(p)+\sigma^2(q)\sigma(\alpha))\alpha)-q(\sigma^2(p)+\sigma^2(q)\sigma(\alpha))\beta}{\Delta}
\end{displaymath}
and this simplifies to 
\begin{displaymath}
\frac{p\,\sigma^2(q)\,\sigma(\beta)\;+\;(p\alpha-q\beta)
\bigl(\sigma^2(p)+\sigma^2(q)\sigma(\alpha)\bigr)}{\Delta}.
\end{displaymath}
The claim follows.
\end{proof}

We can immediately derive a necessary condition for 
similarity of irreducible elements.

\begin{corollary}\label{cor-s1.4-3}
Assume that $x^2-\alpha x -\beta$ 
and $x^2-\alpha' x -\beta'$ are both irreducible.
If $\frac{\beta'}{\beta}\not\in \mathbf{G}$, then 
$M_{\alpha,\beta}\not\cong M_{\alpha',\beta'}$.
\end{corollary}

\begin{proof}
If $M_{\alpha,\beta}\cong M_{\alpha',\beta'}$, then
there exists $(0,0)\neq (p,q)\in\Bbbk^\times$ such that,
for the module $M_{\alpha,\beta}$, we would have
$\mathrm{Ann}_R(p+qx)=R(x^2-\alpha' x -\beta')$.
But, looking at the $(1,2)$-entry in the matrix of 
$[F]_\mathbf{v}^\mathbf{v}$ in Proposition~\ref{prop-s1.4-2},
we obtain $\beta'\in \mathbf{G}\beta$. The claim follows.
\end{proof}

\subsection{Special case: $\alpha=0$}\label{s1.5}

In this subsection, we use elementary means to 
investigate left irreducibility of the elements
of the form $x^2-\beta$ and to study the corresponding simple
$R$-modules $R/R(x^2-\beta)$.

\begin{corollary}\label{cor-s1.5-1}
For $\beta\in \Bbbk^\times$, the element
$x^2-\beta$ is left irreducible if and only if 
$\beta$ is not of the form $u\sigma(u)$, for some
$u\in \Bbbk$.
\end{corollary}

\begin{proof}
This follows directly from Lemma~\ref{lem-s1.4-1}.
\end{proof}

Note that all elements of the form $u\sigma(u)$, where $u\in\Bbbk$,
form a subgroup of the multiplicative group $\Bbbk^\times$.
Denote by $\tilde{\mathbf{G}}$ the subgroup of 
$\mathbf{G}$ consisting of all elements of the form
$\frac{\sigma^2(a)}{a}$
(equivalently, of the form $\frac{a}{\sigma^2(a)}$), where $a\in\Bbbk$. 
Note that we have
$\frac{\sigma^2(a)}{a}=\frac{\sigma^2(a)}{\sigma(a)}\cdot
\frac{\sigma(a)}{a}\in\mathbf{G}$.

\begin{corollary}\label{cor-s1.5-2}
For $\beta,\beta'\in \Bbbk^\times$
such that $x^2-\beta$ and $x^2-\beta'$ are both
left irreducible, we have $R/R(x^2-\beta)\cong R/R(x^2-\beta')$
if and only if $\frac{\beta'}{\beta}\in \tilde{\mathbf{G}}$
or $\frac{\sigma(\beta')}{\beta}\in \tilde{\mathbf{G}}$.
\end{corollary}

\begin{proof}
For $a,b,c,d\in\Bbbk$, the equality
\begin{displaymath}
\left(\begin{array}{cc}a&b\\c&d\end{array}\right)
\left(\begin{array}{cc}0&\beta\\1&0\end{array}\right)=
\left(\begin{array}{cc}0&\beta'\\1&0\end{array}\right)
\left(\begin{array}{cc}\sigma(a)&\sigma(b)\\\sigma(c)&\sigma(d)\end{array}\right)
\end{displaymath}
is equivalent to the following system:
\begin{equation}\label{eq-na25}
\left\{
\begin{array}{rcl}
b&=&\beta'\sigma(c)\\
d&=&\sigma(a)\\
a\beta&=&\beta'\sigma(d)\\
c\beta&=&\sigma(b)
\end{array}
\right.
\end{equation}
If $\left(\begin{array}{cc}a&b\\c&d\end{array}\right)$
is invertible, we have $(a,c)\neq (0,0)$.
If $c\neq 0$, inserting $b=\beta'\sigma(c)$ into the fourth equation
and rewriting gives
\begin{equation}\label{eq-na20}
\frac{\sigma(\beta')}{\beta}=\frac{c}{\sigma^2(c)}. 
\end{equation}
If $a\neq 0$, inserting $d=\sigma(a)$ into the third equation
and rewriting gives
\begin{equation}\label{eq-na21}
\frac{\beta'}{\beta}=\frac{a}{\sigma^2(a)}. 
\end{equation}
This immediately implies the ``only if'' part of the statement.

To prove the ``if'' part, assume $\frac{\beta'}{\beta}\in \tilde{\mathbf{G}}$.
Take $c=b=0$ and choose $a$ such that 
Equation~\eqref{eq-na21} is satisfied. Take $d=\sigma(a)$. 
This gives a solution to the system in Equation~\eqref{eq-na25}
and hence implies $R/R(x^2-\beta)\cong R/R(x^2-\beta')$.

Similarly, assume $\frac{\sigma(\beta')}{\beta}\in \tilde{\mathbf{G}}$.
Take $a=d=0$ and choose $c$ such that 
Equation~\eqref{eq-na20} is satisfied. Take $b=\beta'\sigma(c)$. 
This gives a solution to the system in Equation~\eqref{eq-na25}
and hence implies $R/R(x^2-\beta)\cong R/R(x^2-\beta')$.
The proof is complete.
\end{proof}

Using the above, we can classify the similarity classes of the 
irreducible elements of the form $x^2-\beta$. Fix a real number $\omega$
and let $\Omega_\omega$ denote the set of all complex numbers
whose real part belongs to the half-interval $[\omega,\omega+2)$.
We also denote by $\Omega_\omega^{(2)}$ the set of all complex numbers 
whose real part belongs to the half-interval $[\omega,\omega+4)$.
Let $\mathbf{F}(\Omega_\omega,\mathbb{Z})$ be the set of all functions
from $\Omega_\omega$ to $\mathbb{Z}$. Let 
$\mathbf{F}^f(\Omega_\omega,\mathbb{Z})$ be the subset of
$\mathbf{F}(\Omega_\omega,\mathbb{Z})$ consisting of all functions
for which the {\em support} of the function, that is the set of all
points at which the value of the function is non-zero, is finite.

Let $\Psi$ denote the set of all triples $(c,\zeta,\xi)$, where
$c\in\mathbb{C}^\times$ and 
$\zeta,\xi\in \mathbf{F}^f(\Omega_\omega,\mathbb{Z})$.
To each such triple $(c,\zeta,\xi)$, we associate the rational
function
\begin{equation}\label{eq-defgczx}
g_{(c,\zeta,\xi)}(\mathtt{h}):=
c\cdot \prod_{\lambda\in \Omega_\omega}(\mathtt{h}-\lambda)^{\zeta(\lambda)}
\cdot \prod_{\mu\in \Omega_\omega}(\mathtt{h}-\mu-2)^{\xi(\mu)}.
\end{equation}
The following statement classifies the similarity classes of left 
irreducible elements of the form $x^2-\beta$. 

\begin{theorem}\label{thm-s1.5-3}
We have:

\begin{enumerate}[$($a$)$]
\item\label{thm-s1.5-3.1} For
$(c,\zeta,\xi)\in \Psi$, the element $x^2-g_{(c,\zeta,\xi)}\in R$
is left irreducible if and only if $\zeta\neq \xi$.
\item\label{thm-s1.5-3.2} For
$(c,\zeta,\xi),(c',\zeta',\xi')\in \Psi$
such that both $x^2-g_{(c,\zeta,\xi)}$ and 
$x^2-g_{(c',\zeta',\xi')}$ are left irreducible, we have
$x^2-g_{(c,\zeta,\xi)}\sim x^2-g_{(c',\zeta',\xi')}$
if and only if $(c',\zeta',\xi')=(c,\zeta,\xi)$
or $(c',\zeta',\xi')=(c,\xi,\zeta)$.
\item\label{thm-s1.5-3.3} If $\beta\in\Bbbk$ is such that
$x^2-\beta$ is left irreducible, then there is 
$(c,\zeta,\xi)\in \Psi$ such that 
$x^2-\beta\sim x^2-g_{(c,\zeta,\xi)}$.
\end{enumerate}
\end{theorem}

\begin{proof}
By Corollary~\ref{cor-s1.5-1}, the element 
$x^2-g_{(c,\zeta,\xi)}\in R$ is irreducible
if and only if it is not of the form 
$\sigma(u)u$, for any $u\in\Bbbk$. From the definition of 
$g_{(c,\zeta,\xi)}$ in Formula~\eqref{eq-defgczx}, it follows
directly that $g_{(c,\zeta,\xi)}=\sigma(u)u$ if and only if
$\zeta=\xi$. This proves Claim~\eqref{thm-s1.5-3.1}.

Next, note that $\tilde{\mathbf{G}}$ contains, by definition, 
all elements of the form 
$r_\lambda:=\frac{\mathtt{h}-\lambda}{\mathtt{h}-\lambda-4}$, 
where $\lambda\in\mathbb{C}$. Given a rational function $f\in\Bbbk$,
we can use multiplication and division with such elements to
move all zeros and poles of $f$ into $\Omega_\omega^{(2)}$. In other
words, there exist complex numbers $\lambda_1,\lambda_2,\dots,\lambda_k$
and $\mu_1,\mu_2,\dots,\mu_m$ such that all poles and zeros of 
\begin{displaymath}
\tilde{f}:=\frac{r_{\lambda_1}r_{\lambda_2}\cdots r_{\lambda_k}}{
r_{\mu_1}r_{\mu_2}\cdots r_{\mu_m}}f 
\end{displaymath}
belong to $\Omega_\omega^{(2)}$. This means that $\tilde{f}$ 
equals to $g_{(c,\zeta,\xi)}$, for some 
$(c,\zeta,\xi)\in \Psi$. If $x^2-f$ is irreducible, we have
$x^2-f\sim x^2-\tilde{f}$ by Corollary~\ref{cor-s1.5-2}.
This proves Claim~\eqref{thm-s1.5-3.3}.

Finally, let us prove Claim~\eqref{thm-s1.5-3.2}.
Let $(c,\zeta,\xi)\in \Psi$ be such that $\zeta\neq \xi$.
From Formula~\eqref{eq-defgczx} it follows directly
that $\frac{\sigma(g_{(c,\zeta,\xi)})}{g_{(c,\xi,\zeta)}}\in\tilde{\mathbf{G}}$
and hence $x^2-g_{(c,\zeta,\xi)}\sim x^2-g_{(c,\xi,\zeta)}$
by Corollary~\ref{cor-s1.5-2}.

Let now  $(c',\zeta',\xi')\in \Psi$ be such that 
$x^2-g_{(c,\zeta,\xi)}\sim x^2-g_{(c',\zeta',\xi')}$.
Since all elements in $\tilde{\mathbf{G}}$ are monic, we
immediately obtain $c=c'$ by applying Corollary~\ref{cor-s1.5-2}.
Also, by the same Corollary~\ref{cor-s1.5-2}, we have either 
$\frac{g_{(c,\zeta,\xi)}}{g_{(c',\zeta',\xi')}}\in \tilde{\mathbf{G}}$
or $\frac{\sigma(g_{(c,\zeta,\xi)})}{g_{(c',\zeta',\xi')}}\in \tilde{\mathbf{G}}$.

If $\frac{g_{(c,\zeta,\xi)}}{g_{(c',\zeta',\xi')}}\in \tilde{\mathbf{G}}$, we note
that all zeros and poles of $\frac{g_{(c,\zeta,\xi)}}{g_{(c',\zeta',\xi')}}$
belong to $\Omega_\omega$ due to Formula~\eqref{eq-defgczx}.
At the same time, the only element in $\tilde{\mathbf{G}}$
which has all zeros and poles inside $\Omega_\omega$ is the identity element
(since $\Omega_\omega$ does not contain any pair of complex numbers whose
real parts differ by $2$).
Hence $(c,\zeta,\xi)=(c',\zeta',\xi')$ in this case.

If $\frac{\sigma(g_{(c,\zeta,\xi)})}{g_{(c',\zeta',\xi')}}\in \tilde{\mathbf{G}}$, 
we note that 
$\frac{g_{(c,\xi,\zeta)}}{\sigma(g_{(c,\zeta,\xi)})}\in \tilde{\mathbf{G}}$
by the above and hence
$\frac{g_{(c,\xi,\zeta)}}{g_{(c',\zeta',\xi')}}\in \tilde{\mathbf{G}}$.
This leads to $(c,\xi,\zeta)=(c',\zeta',\xi')$ by the same
argument as in the previous paragraph.
This completes the proof of Claim~\eqref{thm-s1.5-3.2} and hence of the 
whole theorem.
\end{proof}

We note that the algebra $R$ has a natural $\mathbb{Z}$-grading
defined as follows:
\begin{displaymath}
\deg(x)=1,\quad  \deg(x^{-1})=-1\quad  \deg(\mathtt{h})=0.
\end{displaymath}
For any positive integer $k$, this induces 
a $\mathbf{C}_k=\mathbb{Z}/k\mathbb{Z}$-grading. In our case, that is when
$\alpha=0$, the module $M_{0,\beta}$ admits a natural $\mathbf{C}_2$-graded 
lift by putting $v$ in degree $\overline{0}$ and $x(v)$ in degree 
$\overline{1}$. In general, a module $M_{\alpha,\beta}$ 
admits a $\mathbb{Z}_2$-graded  lift if and only if 
$M_{\alpha,\beta}\cong M_{0,\beta'}$, for some $\beta'$.
Therefore Theorem~\ref{thm-s1.5-3} can be interpreted as 
classification of simple $\mathbf{C}_2$-gradable $R$-modules
that have rank $2$ over $\mathbb{C}(\mathtt{h})$.

\subsection{On the irreducibility of $x^2-x-\beta$}\label{s1.7}

In this subsection we discuss irreducibility of 
the element of the form $x^2-x-\beta$. Although we do not
provide an explicit description of all such elements,
we obtain some interesting and non-trivial partial results.
Our aim is to show that this problem of classification of 
irreducible elements of the form $x^2-x-\beta$ is non-trivial
and interesting.

Let $u\in\Bbbk^\times$, written $u=c\frac{f(t)}{g(t)}$, where
$c\in\mathbb{C}^\times$ and $f,g\in\mathbb{C}[t]^\times$ 
are monic and coprime. Then the {\em degree at infinity}
of $u$ is defined as $\mathbf{d}_\infty(u):=\deg(f)-\deg(g)$.
The following general claim, as well as an essential part of the proof,
was suggested by Microsoft Copilot.

\begin{proposition}\label{prop-s1.7-1}
Let $\beta\in\Bbbk^\times$ be such that 
$\mathbf{d}_\infty(\beta)>0$ is odd. Then 
$x^2-x-\beta$ is irreducible.
\end{proposition}

\begin{proof}
By Lemma~\ref{lem-s1.4-1}, we need to show that 
Equation~\eqref{eq-irreq} has no solutions, for
$\alpha=1$ and our $\beta$. We rewrite it as
$y(\sigma(y)-1)=\beta$ and let 
$y=c\frac{f(t)}{g(t)}$, where
$c\in\mathbb{C}^\times$ and $f,g\in\mathbb{C}[t]^\times$ 
are monic and coprime. Inserting this into our equation, we obtain
\begin{displaymath}
c\cdot \frac{f(t)\big(c\cdot 
f(\mathtt{h}-2)-g(\mathtt{h}-2)\big)}{g(\mathtt{h})g(\mathtt{h}-2)}=\beta. 
\end{displaymath}

Assume that $y$ is a solution.
If $\deg(f)>\deg(g)$, then the degree of the numerator 
on the left is
$2\deg(f)$ and this is strictly bigger than the degree $2\deg(g)$ of the
denominator. Any cancellation of a factor in the numerator and denominator 
does not change the
degree at infinity, and hence $\mathbf{d}_\infty(\beta)>0$ must be even.

If $\deg(f)<\deg(g)$, then the degree of the numerator 
on the left is $\deg(f)+\deg(g)$ and this is strictly smaller 
than the degree of the denominator $2\deg(g)$. Any cancellation 
does not change the degree at infinity, and hence 
$\mathbf{d}_\infty(\beta)<0$.

In the case $\deg(f)=\deg(g)$, the degree of the numerator 
on the left is at most  $2\deg(f)=2\deg(g)$ and hence does not exceed
the degree $2\deg(g)$ of the denominator. Any cancellation 
does not change the degree at infinity, and hence 
$\mathbf{d}_\infty(\beta)\leq 0$. The claim follows.
\end{proof}

\subsection{Left irreducible elements of the form  
$x^m-\beta$}\label{s1.8}

In this subsection, we generalize Theorem~\ref{thm-s1.5-3}
using a more sophisticated approach.

Let $m$ be a positive integer.
We fix $\omega\in\mathbb{R}$.
Let $\Psi^{(m)}$ denote the set of all tuples 
$\boldsymbol{\zeta}:=(c,\zeta_0,\zeta_1,\dots,\zeta_{m-1})$, where
$c\in\mathbb{C}^\times$ and all
$\zeta_i\in \mathbf{F}^f(\Omega_\omega,\mathbb{Z})$.
To each such tuple $\boldsymbol{\zeta}$, we associate the rational
function
\begin{displaymath}
g_{\boldsymbol{\zeta}}(\mathtt{h}):=
c\cdot 
\prod_{i=0}^{m-1}
\prod_{\lambda\in \Omega_\omega}(\mathtt{h}-\lambda-2i)^{\zeta_i(\lambda)}.
\end{displaymath}
The cyclic group $\mathbf{C}_m\cong \mathbb{Z}/m\mathbb{Z}$
acts on $\Psi^{(m)}$ as follows: $k\in \mathbf{C}_m$
maps $\zeta_i$ to $\zeta_{i+k}$, where the index is taken modulo $m$.
As usual, we will say that $\boldsymbol{\zeta}\in\Psi^{(m)}$
is {\em regular} provided that the orbit 
$\mathbf{C}_m\cdot \boldsymbol{\zeta}$ is regular, that is,
contains $m$ different elements.
The following statement classifies the similarity classes of the 
irreducible elements of the form $x^m-\beta$, which also can be interpreted as 
classification of simple $\mathbf{C}_m$-gradable $R$-modules
that have rank $m$ over $\mathbb{C}(\mathtt{h})$.

\begin{theorem}\label{thm-s1.8-9}
We have:

\begin{enumerate}[$($a$)$]
\item\label{thm-s1.8-9.1} For
$\boldsymbol{\zeta}\in \Psi^{(m)}$, 
the element $x^m-g_{\boldsymbol{\zeta}}\in R$
is left irreducible if and only if the orbit
$\mathbf{C}_m\cdot \zeta$ is regular.
\item\label{thm-s1.8-9.2} For two regular
$\boldsymbol{\zeta},\boldsymbol{\zeta}'\in \Psi^{(m)}$, the
elements $x^m-g_{\boldsymbol{\zeta}}$ and $x^m-g_{\boldsymbol{\zeta}'}$
are similar
if and only if $\boldsymbol{\zeta}\in \mathbf{C}_m\cdot \boldsymbol{\zeta}'$.
\item\label{thm-s1.8-9.3} If $\beta\in\Bbbk$ is such that
$x^m-\beta$ is left irreducible, then there is 
$\boldsymbol{\zeta}\in \Psi^{(m)}$ such that 
$x^m-\beta$ and $x^m-g_{\boldsymbol{\zeta}}$ are similar.
\end{enumerate}
\end{theorem}

\begin{proof}
Denote by $R^{(m)}$ the subalgebra of $R$ generated by
$\Bbbk$ and $x^{\pm m}$. Clearly, we have $R^{(m)}=\Bbbk[y,y^{-1},\sigma^m]$,
where $y=x^m$. For any $\beta\in \Bbbk$, consider 
the $R$-module $M:=R/R(x^m-\beta)$ and its basis
$\{v_i\,:\, i\in\mathbf{C}_m\}$ given by $v_i=x^i+ R(x^m-\beta)$.
In this basis, the action
of $x$ and $x^{-1}$ is given by
\begin{equation}\label{eq-xandxinv2}
\left(\begin{array}{cccccc}
0&0&0&\dots&0&\beta\\
1&0&0&\dots&0&0\\
0&1&0&\dots&0&0\\
\vdots&\vdots&\vdots&\ddots&\vdots&\vdots\\
0&0&0&\dots&0&0\\
0&0&0&\dots&1&0\\
\end{array}\right)\quad\text{ and }\quad 
\left(\begin{array}{cccccc}
0&1&0&\dots&0&0\\
0&0&1&\dots&0&0\\
0&0&0&\dots&0&0\\
\vdots&\vdots&\vdots&\ddots&\vdots&\vdots\\
0&0&0&\dots&0&1\\
\resizebox{9mm}{!}{$\sigma^{-1}(\beta^{-1})$}&0&0&\dots&0&0\\
\end{array}\right),
\end{equation}
respectively. We stress the attention to the fact  that both $x$
and $x^{-1}$ are only  $\mathbb{C}(\mathtt{h})$-linear up to $\sigma$
(resp. $\sigma^{-1}$), which is the reason why it is
$\sigma^{-1}(\beta^{-1})$ and not just $\beta^{-1}$
that appears in the matrix describing the action of  $x^{-1}$.
Denote by $N_i$ the subspace $\Bbbk v_i$ of $M$.
Then 
\begin{equation}\label{eq-simrm}
M= N_0\oplus N_1\oplus\dots\oplus N_{m-1}, 
\end{equation}
by construction.

Now let us analyze the restriction of $M$ to $R^{(m)}$. 
Using \eqref{eq-xandxinv2}, we have
\begin{displaymath}
y\cdot  v_i = x^m\cdot  v_i = \sigma^{i}(\beta) v_i,\quad
\text{ for }i\in \mathbf{C}_m.
\end{displaymath}
Consequently, each $N_i$ is an $R^{(m)}$-submodule.
By construction, $N_i$ has dimension $1$ over $\Bbbk$ 
and hence is simple, as an  $R^{(m)}$-module. Therefore, the decomposition
\eqref{eq-simrm} is a decomposition of $M$ into a direct sum
of simple $R^{(m)}$-submodules.

Now we prove Claim~\eqref{thm-s1.8-9.1}. Let 
$\beta=g_{\boldsymbol{\zeta}}$, for some regular 
$\boldsymbol{\zeta}$. From Subsection~\ref{s1.2}
(more precisely, from  \cite[Section~2]{GKM}), it follows that the 
$R^{(m)}$-modules $N_0,N_1,\dots,N_m$ are pair-wise 
non-isomorphic. Since they all are simple, it follows that,
for any $i\neq j$ and any $0\neq a_i\in N_i$ and 
$0\neq a_j\in N_j$, the corresponding annihilators 
$\mathrm{Ann}_{R^{(m)}}(a_i)$ and $\mathrm{Ann}_{R^{(m)}}(a_j)$
are different maximal left ideals in $R^{(m)}$. In particular,
neither can we have $\mathrm{Ann}_{R^{(m)}}(a_i)\subset 
\mathrm{Ann}_{R^{(m)}}(a_j)$ nor vice versa.

We want to prove that $M$ is simple. 
Let $w\in M$ be a non-zero element. Consider
$Rw$. Each element in $Rw$ can be written in our basis
$\{v_i\}$. Let $w'\in Rw$ be a non-zero element for which 
the number of non-zero coefficients with respect to this
basis is minimal possible. Assume that the number of 
these non-zero coefficients is at least two. Take two
such coefficients, call one of them $a$ and assume it is the
coefficient at $v_i$, call the other one $b$ and assume
it is the coefficient at $v_j$, where $i\neq j$. 
Now, by the previous paragraph, there is 
$u\in \mathrm{Ann}_{R^{(m)}}(av_i)\setminus 
\mathrm{Ann}_{R^{(m)}}(bv_j)$. Applying $u$ to $w'$, we certainly
kill the coefficient at $v_i$ and we do not kill the coefficient
at $v_j$. Also, no new non-zero coefficients can be created.
Therefore $u\cdot w'\neq 0$ has strictly fewer non-zero coefficients
compared to $w'$, a contradiction. This implies that 
$w'=bv_j$, for some $j\in\mathbf{C}_m$ and some $b\in\Bbbk^\times$.
Applying $b^{-1}$, we get $v_j\in Rw$. Applying $x$ repeatedly
(and, when necessary, $\beta^{-1}$), we get that 
$Rw$ contains all basis elements in $\{v_i\}$ and hence coincides with $M$.
This proves the ``if'' part of Claim~\eqref{thm-s1.8-9.1}.

To prove the ``only if'' part, let 
$\beta=g_{\boldsymbol{\zeta}}$, for some 
$\boldsymbol{\zeta}$ which has a non-trivial stabilizer $G$ in
$\mathbf{C}_m$. As any subgroup of a cyclic group is cyclic,
$G$ is generated by some $1<r<m$ which divides $m$. Set
\begin{displaymath}
\tilde{g}:=
\tilde{c}\cdot 
\prod_{i=0}^{r-1}
\prod_{\lambda\in \Omega_\omega}(\mathtt{h}-\lambda-2i)^{\zeta_i(\lambda)},
\end{displaymath}
where $(\tilde{c})^{k}=c$, for $k=\frac{m}{r}$. 
Note that $G\cong\mathbf{C}_k$. Then, directly from 
the definitions, we have
\begin{displaymath}
g_{\boldsymbol{\zeta}}=\tilde{g}
\sigma^{r}(\tilde{g})\sigma^{2r}(\tilde{g})\dots
\sigma^{(k-1)r}(\tilde{g}).
\end{displaymath}
This implies that 
\begin{equation}\label{eq-decomp1}
x^m-g_{\boldsymbol{\zeta}}=A\big(x^r- \tilde{g}\big), 
\end{equation}
where $A$ equals
\begin{multline*}
x^{(k-1)r}+
\sigma^{(k-1)r}(\tilde{g})x^{(k-2)r}
+\big(\sigma^{(k-1)r}(\tilde{g})\sigma^{(k-2)r}(\tilde{g})\big)x^{(k-3)r}
+\dots \\ \dots +
\big(\sigma^{(k-1)r}(\tilde{g})\dots \sigma^{2r}(\tilde{g})\big)x^{r}
+
\big(\sigma^{(k-1)r}(\tilde{g})\dots \sigma^{r}(\tilde{g})\big).
\end{multline*}
Consequently, $x^m-g_{\boldsymbol{\zeta}}$ is not left irreducible,
completing the proof of Claim~\eqref{thm-s1.8-9.1}.

Claim~\eqref{thm-s1.8-9.3} is proved similarly to the 
proof of Theorem~\ref{thm-s1.5-3}\eqref{thm-s1.5-3.3}.
For $\alpha\in\Bbbk^\times$, consider the basis 
$\{w_i:=\sigma^i(\alpha)v_i\,:\,i\in\mathbb{C}_m\}$ of $M$ 
and note that $x\cdot w_i =w_{i+1}$ if $i\neq m-1$
and $x\cdot w_{m-1}=\frac{\sigma^m(\alpha)}{\alpha}\beta w_0$.
This implies that $R/R(x^m-\beta )$
and $R/R(x^m-\beta\frac{\sigma^m(\alpha)}{\alpha} )$ are isomorphic.
Now, to prove Claim~\eqref{thm-s1.8-9.3}, we do induction on 
the sum of distances between the real parts of  
all poles and zeros of $\beta$ (counted with their multiplicities) 
to the half-interval 
$[\omega,\omega+2m)$. If the sum is zero, 
we have  $\beta=g_{\boldsymbol{\zeta}}$,
for some $\boldsymbol{\zeta}\in \Psi^{(m)}$, 
so we have nothing to prove.

Now assume that $\beta$ has a zero or pole $\lambda$
whose real part does not belong to $[\omega,\omega+2m)$.
Then the real part of either $\lambda+2m$ or 
$\lambda-2m$ is closer to $[\omega,\omega+2m)$. Taking
either $\alpha=\mathtt{h}-\lambda$ or $\alpha=(\mathtt{h}-\lambda)^{-1}$
or $\alpha=\mathtt{h}-\lambda-2m$ or $\alpha=(\mathtt{h}-\lambda-2m)^{-1}$
(depending on whether we have a zero or a pole and
whether $\lambda-2m$ or $\lambda+2m$ is closer to $[\omega,\omega+2m)$)
and applying the above base change replaces 
$\beta$ with $\beta\frac{\sigma^m(\alpha)}{\alpha}$
and lowers our sum by $2m$. Now 
Claim~\eqref{thm-s1.8-9.3} follows by induction.

For Claim~\eqref{thm-s1.8-9.2}, we observe that existence of
an isomorphism between $R/R(x^m-g_{\boldsymbol{\zeta}})$
and $R/R(x^m-g_{\boldsymbol{\zeta'}})$ implies an isomorphism
between their restrictions to $R^{(m)}$. Such restrictions
are explicitly described in the first paragraph of this proof.
When we restrict $R/R(x^m-g_{\boldsymbol{\zeta}})$ to 
$R^{(m)}$, we get a direct sum of the $N_i$'s, where
$N_i$ is the one-dimensional (over $\Bbbk$) simple
$R^{(m)}$-module corresponding to $\sigma^i(g_{\boldsymbol{\zeta}})$.

Note that the real parts of the zeros and poles of 
$\sigma^i(g_{\boldsymbol{\zeta}})$ get shifted from 
$[\omega,\omega+2m)$, for the element $g_{\boldsymbol{\zeta}}$, to 
$[\omega+2i,\omega+2m+2i)$, for the element $\sigma^i(g_{\boldsymbol{\zeta}})$.
We can now use the above trick of replacing
$\beta$ by $\beta\frac{\sigma^m(\alpha)}{\alpha}$
and move the zeros and poles which landed outside of 
our $[\omega,\omega+2m)$-strip back. This does not affect
our module, up to isomorphism, as explained above. The 
effect on $g_{\boldsymbol{\zeta}}$ corresponds to the cyclic 
permutation of the components of $\zeta$ in the sense that,
for any $i\in\mathbf{C}_m$,
we have an isomorphism
\begin{displaymath}
R^{(m)}/ R^{(m)}(y-\sigma^i(g_{\boldsymbol{\zeta}}))
\cong R^{(m)}/ R^{(m)}(y-g_{(m-i)\cdot \boldsymbol{\zeta}})).
\end{displaymath}
In fact, the same trick also works at the level of $R$-modules,
which implies that 
\begin{displaymath}
R/ R(x^m-\sigma^i(g_{\boldsymbol{\zeta}}))
\cong R/ R(x^m-g_{(m-i)\cdot \boldsymbol{\zeta}})).
\end{displaymath}
This implies the ``if'' part of Claim~\eqref{thm-s1.8-9.2}.

To prove the ``only if'' part, we note that, by
Subsection~\ref{s1.2}, if we have $\beta$ and $\beta'$
such that all real parts of all poles and zeros of these
two elements belong to $[\omega,\omega+2m)$, then 
the $R^{(m)}$-modules
$R^{(m)}/R^{(m)}(y-\beta)$
and $R^{(m)}/R^{(m)}(y-\beta')$ are isomorphic if and only
if $\beta=\beta'$. Therefore, our previous discussion 
implies that  an isomorphism between 
$R/R(x^m-g_{\boldsymbol{\zeta}})$ and $R/R(x^m-g_{\boldsymbol{\zeta'}})$
means that $g_{\boldsymbol{\zeta'}}$ can be obtained from 
$g_{\boldsymbol{\zeta}}$ by some $\sigma^i$-shift followed by
the cyclic reshuffling of the zeros and poles, as described above. 
This means exactly
that $\boldsymbol{\zeta}=i\cdot \boldsymbol{\zeta'}$,
that is,
$\boldsymbol{\zeta}\in \mathbf{C}_m\cdot \boldsymbol{\zeta}'$.
This proves Claim~\eqref{thm-s1.8-9.2} and completes the proof.
\end{proof}

\begin{example}\label{example13}
{\em 
For any positive integer $m$, the element $x^m-\mathtt{h}$
is left irreducible. 
}
\end{example}

\begin{corollary}\label{cor-lirr}
For $m\in\mathbb{Z}_{>0}$ and $\beta\in\Bbbk^\times$, the element
$x^m-\beta$ is left irreducible if and only if, for any prime
divisor $p\vert m$, we cannot write 
$\beta$ as $\alpha\sigma^{\frac{m}{p}}(\alpha)\sigma^{2\frac{m}{p}}(\alpha)
\dots \sigma^{(p-1)\frac{m}{p}}(\alpha)$,
where $\alpha\in\Bbbk^\times$.
\end{corollary}

\begin{proof}
If $\beta=
\alpha\sigma^{\frac{m}{p}}(\alpha)\sigma^{2\frac{m}{p}}(\alpha)
\dots \sigma^{(p-1)\frac{m}{p}}(\alpha)$, the element
$x^m-\beta$  admits a decomposition as
in \eqref{eq-decomp1} and hence it is not left irreducible.

On the other hand, if $x^m-\beta$ is left reducible, 
then, to start with, $m>1$. Then,  by 
Theorem~\ref{thm-s1.8-9}\eqref{thm-s1.8-9.3}, 
there is $\gamma\in\Bbbk^\times$
such that $\frac{\sigma^{m}(\gamma)}{\gamma}\beta=
g_{\boldsymbol{\zeta}}$, for some 
$\boldsymbol{\zeta}\in \Psi^{(m)}$ which is not regular. 
This means that the stabilizer of $\boldsymbol{\zeta}$
in $\mathbf{C}_m$ is non-trivial, say, 
equals $r\mathbf{C}_m$, for some proper divisor $r\vert m$,
that is $1\leq r<m$.

Consider the element
\begin{displaymath}
\xi=c'\cdot 
\prod_{i=0}^{r-1}
\prod_{\lambda\in \Omega_\omega}(\mathtt{h}-\lambda-2i)^{\zeta_i(\lambda)},
\end{displaymath}
where $c'\in\mathbb{C}$ is such that $(c')^{\frac{m}{r}}=c$.
From the construction of $g_{\boldsymbol{\zeta}}$, we obtain that 
\begin{displaymath}
g_{\boldsymbol{\zeta}} =\xi\sigma^{r}(\xi)\sigma^{2r}(\xi)\dots
\sigma^{(\frac{m}{r}-1)r}(\xi).
\end{displaymath}
If we take $p$ to be any prime divisor of $\frac{m}{r}$, we can
combine the factors in the letter product and write this product in the form
\begin{displaymath}
g_{\boldsymbol{\zeta}} =\xi'\sigma^{\frac{m}{p}}(\xi')
\sigma^{2\frac{m}{p}}(\xi')\dots
\sigma^{(p-1)\frac{m}{p}}(\xi'),
\end{displaymath}
for an appropriate $\xi'$.

Also note that 
\begin{displaymath}
\frac{\gamma}{\sigma^m(\gamma)}=
\frac{\gamma{\color{violet}\sigma(\gamma)\cdots \sigma^{m-1}(\gamma)}}{
{\color{violet}\sigma(\gamma)\cdots \sigma^{m-1}(\gamma)}\sigma^m(\gamma)}
\end{displaymath}
and here the factors on the right hand side can also be combined so that
the whole expression is written in the form 
\begin{displaymath}
\xi''\sigma^{\frac{m}{p}}(\xi'')
\sigma^{2\frac{m}{p}}(\xi'')\dots
\sigma^{(p-1)\frac{m}{p}}(\xi''),
\end{displaymath}
for an appropriate $\xi''$. Combining the two expressions, we 
obtain that $\beta$ has the necessary form as claimed.
This completes the proof.
\end{proof}

\begin{example}\label{ex-irr}
{\em
From Corollary~\ref{cor-lirr}, we have that 
an element $x^4-\beta$ is left irreducible if and only if 
$\beta$ is not of the form $\alpha\sigma^2(\alpha)$, for any
$\alpha\in \Bbbk$. This seems to be a significant improvement of
\cite[Corollary~2]{BP}.
}
\end{example}

\section{Applications to $\mathfrak{sl}_2$-modules}\label{s2}

\subsection{$\mathfrak{sl}_2$-setup}\label{s2.1}

For a Lie algebra  $\mathfrak{a}$, we will  denote the universal enveloping
algebra of $\mathfrak{a}$ by $U(\mathfrak{a})$. 

We refer the reader to \cite{Ma10} for more details on 
the Lie algebra $\mathfrak{sl}_2$ and its representations.
Let $\mathfrak{g}$ be the Lie algebra $\mathfrak{sl}_2$ over $\mathbb{C}$.
We fix the standard basis
\begin{displaymath}
\mathtt{e}:=\left(\begin{array}{cc}0&1\\0&0\end{array}\right),\quad
\mathtt{f}:=\left(\begin{array}{cc}0&0\\1&0\end{array}\right),\quad
\mathtt{h}:=\left(\begin{array}{cc}1&0\\0&-1\end{array}\right)
\end{displaymath}
in $\mathfrak{g}$. Then the Lie bracket on $\mathfrak{g}$ is 
given by
\begin{displaymath}
[\mathtt{h},\mathtt{e}]=2\mathtt{e},\quad 
[\mathtt{h},\mathtt{f}]=-2\mathtt{f},\quad 
[\mathtt{e},\mathtt{f}]=\mathtt{h}. 
\end{displaymath}
In particular, we have the following relations:
\begin{equation}\label{eq1}
\mathtt{h}\mathtt{e}=\mathtt{e}(\mathtt{h}+2)\quad
\text{ and }
\mathtt{h}\mathtt{f}=\mathtt{f}(\mathtt{h}-2) 
\end{equation}
in $U(\mathfrak{g})$. We consider the Casimir element
\begin{equation}\label{eq1a}
\mathtt{c}:=(\mathtt{h}+1)^2+4 \mathtt{f}\mathtt{e}=
(\mathtt{h}-1)^2+4 \mathtt{e}\mathtt{f}\in U(\mathfrak{g})
\end{equation}
which generates the center of $U(\mathfrak{g})$.
Given $\vartheta\in\mathbb{C}$, we will denote by $U_\vartheta$
the quotient algebra $U(\mathfrak{g})/U(\mathfrak{g})(\mathtt{c}-\vartheta)$.

\subsection{$U_\vartheta$ and skew Laurent polynomial algebra}\label{s2.3}

Connecting with Section~\ref{s1}, we consider the field 
$\Bbbk:=\mathbb{C}(\mathtt{h})$ as well as its
automorphism $\sigma$ defined as the identity on $\mathbb{C}$ and,
additionally, via $\sigma(\mathtt{h})=\mathtt{h}-2$. 
This gives us the corresponding 
skew Laurent polynomial algebra $R:=\Bbbk[x,x^{-1},\sigma]$
which we studied in Section~\ref{s1}.

By \cite[Theorem~6.2.12]{Ma10}, if we fix $\vartheta\in\mathbb{C}$, 
then we have an injective algebra homomorphism $\Phi$ from
$U_\vartheta$ to $R$ which is uniquely determined by
\begin{displaymath}
\Phi(\mathtt{e})=x\quad\text{ and }\quad
\Phi(\mathtt{f})=
\frac{\vartheta-(\mathtt{h}+1)^2}{4}x^{-1}\quad\text{ and }\quad
\Phi(\mathtt{h})=\mathtt{h}.
\end{displaymath}
We have the restriction  functor
\begin{displaymath}
\mathrm{Res}^{R}_{U_\vartheta}:R\text{-Mod}\to  U_\vartheta\text{-Mod}
\end{displaymath}
defined by pulling back via $\Phi$. For simplicity,
we will often identify $U_\vartheta$ with its image under 
the homomorphism $\Phi$.

By \cite{Bl79,Bl81}, if $M$ is a simple $R$-module, then the
module $\mathrm{Res}^{R}_{U_\vartheta}(M)$ has simple socle.
Moreover, the map $M\mapsto \mathrm{Soc}(\mathrm{Res}^{R}_{U_\vartheta}(M))$
is a bijection between the set of the isomorphism classes of simple
$R$-modules and the set of the isomorphism classes of simple
$U_\vartheta$-modules that are torsion free over $U(\mathtt{h})$.
In \cite{GKM}, this observation, combined with the description in
Subsection~\ref{s1.2}, was used to provide an explicit 
classification of simple $\mathfrak{g}$-modules that are 
torsion free over $U(\mathtt{h})$ of rank $1$.

\subsection{Explicit description for $x^m-\beta$}\label{s2.5}

Let $m$ be a positive integer and $\beta\in \Bbbk$.
Consider the $R$-module $N_\beta:=R/R(x^m-\beta)$. 
In this subsection we provide an explicit description of the 
$\mathfrak{g}$-modules which appear as socles of the modules
$\mathrm{Res}^{R}_{U_\vartheta}(N_{\beta})$ in the case 
when $N_{\beta}$ is a simple $R$-module. Here $\vartheta\in\mathbb{C}$
is fixed. To simplify our arguments, we make the following choice
which depends on $\vartheta$:
we choose $\omega$ such that the real parts of all roots of
the polynomial $\frac{\vartheta-(\mathtt{h}+1)^2}{4}$ are strictly
smaller than $\omega$. Let $r_1,r_2\in\mathbb{C}$ be such that
$\vartheta-(\mathtt{h}+1)^2=-(\mathtt{h}-r_1)(\mathtt{h}-r_2)$.
Note that $r_1=r_2$ is possible and, if this is indeed
the case, we have $r_1=r_2=-1$. 

Now, fix $c\in\mathbb{C}^\times$ and 
$\zeta_0,\zeta_1,\dots,\zeta_{m-1}\in \mathbf{F}^{f}(\Omega_\omega,\mathbb{Z})$
such that the corresponding $\boldsymbol{\zeta}=(c,\zeta_0,\dots,\zeta_{m-1})$ is regular.  
Then, for $\beta=g_{\boldsymbol{\zeta}}$, the $R$-module $N_{\beta}$ 
is simple, see Theorem~\ref{thm-s1.8-9}. 
Denote by $L_\beta$ the 
simple socle of $\mathrm{Res}^{R}_{U_\vartheta}(N_{\beta})$.
Our aim is to determine a basis of $L_\beta$ explicitly.
Note that $L_\beta\neq N_{\beta}$ because $L_\beta$
has countable dimension while $N_{\beta}$ has uncountable
dimension.

Let $\mathtt{m}:\Omega_\omega^{(m)}\to\mathbb{Z}$ be the function
such that 
\begin{displaymath}
\beta=c\cdot \prod_{\lambda\in \Omega_\omega^{(m)}}
(\mathtt{h}-\lambda)^{\mathtt{m}(\lambda)}.
\end{displaymath}

Write $\beta=c\cdot \frac{a(\mathtt{h})}{b(\mathtt{h})}$, where 
$a,b\in \mathbb{C}[\mathtt{h}]$ are monic and coprime.
Note that $a,b\neq 0$. Set $v_i:=x^i+R(x^m-\beta)$,
for $i=0,1,\dots,m-1$. Then $v_0,v_1,\dots,v_{m-1}$
is a  basis of  $N_{\beta}$. Note that this basis is homogeneous
with respect to the $\mathbf{C}_m$-grading
(the degree of $\mathtt{e}$ is $1$, the degree of $\mathtt{f}$ is $-1$
and the degree of $\mathtt{h}$ is $0$).
In what follows,
the notation $\sum_i$ will mean that we sum over all
indices of this basis. 
To simplify our notation, for $r\in\Bbbk$ and $i\in\mathbb{Z}_{>0}$, 
we denote by $[r]^{(i)}$ the element $r\sigma(r)\sigma^2{r}\cdots\sigma^{i-1}(r)$
and by $[r]_{(i)}$ the element $r\sigma^{-1}(r)\sigma^{-2}{r}\cdots\sigma^{1-i}(r)$.

\begin{lemma}\label{lem-s2.5-1}
For each $j\in\{0,1,\dots,m-1\}$, there is a non-zero 
$f_j\in\mathbb{C}[\mathtt{h}]$ such that 
$L_\beta$ contains $f_jv_j$.
\end{lemma}

\begin{proof}
The action of $\mathbb{C}[\mathtt{h}]$ is torsion-free, in particular,
starting from any non-zero element of  $\Bbbk$, we can multiply it
with a non-zero element of $\mathbb{C}[\mathtt{h}]$ to get 
a non-zero element of $\mathbb{C}[\mathtt{h}]$. Therefore 
$L_\beta$ contains a non-zero element of the form $\sum_i r_i v_i$,
where all $r_i\in \mathbb{C}[\mathtt{h}]$. If only one $r_i$ is non-zero,
then $L_\beta$ contains the corresponding $r_i v_i$. 
The element $bx$ clearly belongs to the image of $\mathfrak{sl}_2$.
Applying it to $r_i v_i$ repeatedly, we get non-zero elements of
the form $f_jv_j$, for all $j$.

If more than one coefficient $r_i$ is non-zero, we need to show that 
$L_\beta$ contains some other non-zero element in which 
a strictly smaller number of coefficients is non-zero.
Then we will be done by induction.

The element $(bx)^m$ belongs to the image of $\mathfrak{sl}_2$
as $bx$ belongs to this image. In our basis, the action of
$(bx)^m$ is diagonal and given by:
\begin{equation}\label{eq-bxm3}
(bx)^m\cdot v_i = 
[b]^{(m)} \sigma^i(\beta)v_i.
\end{equation}
Consequently, $L_\beta$ contains 
a non-zero element of the form 
$\sum_i [b]^{(m)} \sigma^m(r_i) \sigma^i(\beta)v_i$.
Let $s<t$ be two indices such that $r_s\neq 0$
and $r_{t}\neq 0$. Consider the following linear combination:
\begin{equation}\label{eq-lincomb5}
([b]^{(m)} \sigma^m(r_s) \sigma^s(\beta))\cdot \sum_i r_i v_i
-
(r_s)\cdot \sum_i [b]^{(m)} \sigma^m(r_i) \sigma^i(\beta)v_i
\end{equation}
By construction, the coefficient at $v_s$ vanishes. At the same time,
the coefficient at $v_t$ in this combination is
\begin{displaymath}
[b]^{(m)} \sigma^m(r_s) \sigma^s(\beta)r_t-
r_s [b]^{(m)} \sigma^m(r_t) \sigma^t(\beta). 
\end{displaymath}
This vanishes if and only if 
\begin{displaymath}
\frac{\sigma^m(r_sr_t^{-1})}{r_sr_t^{-1}}=
\frac{\sigma^t(\beta)}{\sigma^s(\beta)} .
\end{displaymath}
Applying $\sigma^{-s}$ and denoting $u:=\sigma^{-s}(r_sr_t^{-1})$,
we equivalently obtain
\begin{displaymath}
\frac{\sigma^m(u)}{u}=\frac{\sigma^{t-s}(\beta)}{\beta}. 
\end{displaymath}
Note that we can write the left hand side as
\begin{displaymath}
\frac{\sigma^m(u)}{u}=
\frac{\sigma(u)}{u}
\frac{\sigma^2(u)}{\sigma(u)}\dots
\frac{\sigma^m(u)}{\sigma^{m-1}(u)}=
\left[\frac{\sigma(u)}{u}\right]^{(m)},
\end{displaymath}
so we have 
\begin{equation}\label{eq-redn1}
\frac{\sigma^{t-s}(\beta)}{\beta}=
\left[\frac{\sigma(u)}{u}\right]^{(m)}.
\end{equation}

Consider the additive group $\mathbb{C}$ and its subgroup
$2\mathbb{Z}$. We interpret the multiplicative group of 
monic rational functions over $\mathbb{C}$ as the additive
group of functions from $\mathbb{C}$ to $\mathbb{Z}$ with 
finite support. Given such $f:\mathbb{C}\to \mathbb{Z}$,
we can ask when does $f$ have the form $\frac{\sigma^m(g)}{g}$,
for some $g$? It is easy to see that this holds if and only 
if the sum of the values of $f$ over every coset 
in $\mathbb{C}/2m\mathbb{Z}$ is zero.

Now consider the monic part of our $\beta$ and let
$f$ be the corresponding function. We want to understand
$\frac{\sigma^{t-s}(\beta)}{\beta}$ (note that this one is
monic independently of $\beta$). The corresponding function
$g$ is the difference between the $2(t-s)$-shift of $f$ and $f$.
By Equation~\eqref{eq-redn1} we know that $g$ has the property that
the sum of the values of $g$ over every coset in $\mathbb{C}/2m\mathbb{Z}$ 
is zero. 

For $\lambda\in\mathbb{C}$, consider $\lambda+2\mathbb{Z}$
and let $\mathbf{s}_i$, for $i\in\{0,1,\dots,m-1\}$, be the sum of the 
values of $f$ over the coset $\lambda+2i+2m\mathbb{Z}$.
Since $0<t-s<m$, we have the equations $\mathbf{s}_i-\mathbf{s}_{i+(t-s)}=0$,
for all $i$, where the indices are understood modulo $m$.
The only solutions to these equations are the following:
for each coset $\chi$ in $\mathbb{C}_{m}/(t-s)\mathbb{C}_{m}$,
there is an integer $k_{\lambda,\chi}$ such that 
$\mathbf{s}_j=k_{\lambda,\chi}$ for $j\in \chi$.

Now note that each coset $\lambda+2m\mathbb{Z}$ has a unique
intersection with $\Omega_\omega^{(m)}$. And each coset
$\lambda+2\mathbb{Z}$ intersects $\Omega_\omega^{(m)}$ at exactly
$m$ points, one for each of the subcosets $\lambda+2i+2m\mathbb{Z}$,
where $i\in\{0,1,\dots,m-1\}$.
Consequently, the condition $\mathbf{s}_j=k_{\lambda,\chi}$, 
where $j\in \chi$ and for all $\lambda$, 
means exactly that the values of $f$ over every subcoset
corresponding to
$\chi$ coincide, for every $\lambda$. This means exactly that
$(t-s)\mathbb{C}_{m}$ belongs to the stabilizer of $\boldsymbol{\zeta}$. 
This contradicts our assumption that $x^m-\beta$
is left irreducible. Therefore Equation~\eqref{eq-redn1} is not possible.
Consequently, our linear combination \eqref{eq-lincomb5}
is non-zero and has a strictly smaller number of non-zero coefficients.
So, now the proof of our lemma is completed by induction
on the number of such non-zero coefficients.
\end{proof}

For $j\in\{0,1,\dots,m-1\}$, denote by $I_j$ the ideal of 
$\mathbb{C}[\mathtt{h}]$ which consists of all element 
$q$  such that $qv_j\in L_\beta$. This ideal is non-zero
by Lemma~\ref{lem-s2.5-1}. Let $q_j$ be the (unique) monic generator of 
$I_j$, that is, $I_j=\mathbb{C}[\mathtt{h}]q_j$.

\begin{lemma}\label{lem-s2.5-2}
All $q_j$ are equal to $1$. 
\end{lemma}

\begin{proof}
Recall from the proof of Lemma~\ref{lem-s2.5-1} that the element
$(bx)^m$ belongs to the image of $U(\mathfrak{sl}_2)$ in $R$.
Applying to $q_jv_j$ the element $(bx)^m$ and using 
Equation~\eqref{eq-bxm3}, we get
\begin{displaymath}
[b]^{(m)}\sigma^j(\beta)\sigma^m(q_j)v_j=c\cdot 
\frac{[b]^{(m)}}{\sigma^j(b)}\sigma^j(a)\sigma^m(q_j)v_j,
\end{displaymath}
where $\frac{[b]^{(m)}}{\sigma^j(b)}\in \mathbb{C}[\mathtt{h}]$.
In particular,  $q_j$ must divide the coefficient
$\frac{[b]^{(m)}}{\sigma^j(b)}\sigma^j(a)\sigma^m(q_j)$.
Applying $(bx)^m$ again and again and using that, eventually,
$q_j$ will be coprime with $\sigma^{mr}(q_j)$, for $r\gg 0$, we obtain
that $q_j$ will divide some element of the form
\begin{displaymath}
u\sigma^m(u)\sigma^{2m}(u)\dots \sigma^{rm}(u),
\text{ where }
u=\frac{[b]^{(m)}}{\sigma^j(b)}\sigma^j(a).
\end{displaymath}
Consequently, the real part of any zeros of $q_j$ is greater than or equal to 
$\omega$.

On the other hand, we can play a similar game with 
the element $(\sigma^{-1}(a)\mathtt{u}x^{-1})^m$, where
$\mathtt{u}=(\vartheta-(\mathtt{h}+1)^2)$. This element
also belongs to the image of $U(\mathfrak{sl}_2)$ in $R$.
In our basis, the action of
$(\sigma^{-1}(a)\mathtt{u}x^{-1})^m$ is also diagonal and given by:
\begin{equation}\label{eq-axim7}
(\sigma^{-1}(a)\mathtt{u}x^{-1})^m\cdot v_i = 
[\sigma^{-1}(a)\mathtt{u}]_{(m)} \sigma^{i-m}(\beta^{-1})v_i.
\end{equation}
Therefore, applying $(\sigma^{-1}(a)\mathtt{u}x^{-1})^m$ to 
$q_jv_j$ once gives
\begin{displaymath}
[\sigma^{-1}(a)\mathtt{u}]_{(m)} \sigma^{j-m}(\beta^{-1})
\sigma^{-m}(q_j)v_j=\frac{1}{c}\cdot
\frac{[\sigma^{-1}(a)\mathtt{u}]_{(m)}}{\sigma^{j-m}(a)}
\sigma^{j-m}(b)\sigma^{-m}(q_j)v_j
\end{displaymath}

So, applying $(\sigma^{-1}(a)\mathtt{u}x^{-1})^{rm}$, for $r\gg 0$,
we obtain that $q_j$ must divide some 
\begin{displaymath}
v\sigma^{-m}(v)\sigma^{-2m}(v)\dots \sigma^{-rm}(v),
\text{ where }
v=\frac{[\sigma^{-1}(a)\mathtt{u}]_{(m)}}{\sigma^{j-m}(a)}
\sigma^{j-m}(b).
\end{displaymath}
Due to our choice of $\omega$, the real part of any 
zero of $q_j$ is strictly less than $\omega+2m$. 
Combining the two conditions, we obtain that the real part of 
any  zero of $q_j$ belongs to the half-interval $[\omega,\omega +2m)$.

Now we apply just $x$, which corresponds to the action of $\mathtt{e}$.
If $j\neq m-1$, applying $x$ to $q_jv_j$ gives $\sigma(q_j)v_{j+1}$
and hence $q_{j+1}$ must divide $\sigma(q_j)$. This implies, by induction
on $j$, that the real part of any zero of $q_j$ is greater than or
equal to $\omega+2j$, for any $j$. Playing the same game with 
$\mathtt{u}x^{-1}$, which corresponds to the action of $\mathtt{f}$
(up to a non-zero scalar), we see that $q_jv_j$ is sent to
$\mathtt{u}\sigma^{-1}(q_j)v_{j-1}$, if $j\neq 0$. 
As we already know that all $q_j$ are coprime with 
$\mathtt{u}$ (the latter has zeros outside $\Omega_\omega^{(m)}$
by our choice of $\omega$), it follows that 
$q_{j-1}$ must divide $\sigma^{-1}(q_j)$. This implies, by induction
on $j$, that the real part of any zero of $q_j$ is less than 
$\omega+2j+2$, for any $j$. To sum it up: the real part of
any zero of each $q_j$ belongs to $[\omega+2j,\omega+2j+2)$.

Now we apply to $q_0v_0$ the element 
$\sigma^{-1}(a)\mathtt{u}x^{-1}$  resulting in
$\sigma^{-1}(q_0)\sigma^{-1}(b)\mathtt{u}v_{m-1}$.
So, $q_{m-1}$ must divide $\sigma^{-1}(q_0)\sigma^{-1}(b)\mathtt{u}$.
All zeros of $q_0$ have real parts in $[\omega,\omega+2)$
and hence all zeros of $\sigma^{-1}(q_0)$ are outside
$\Omega_\omega^{(m)}$. So, $\sigma^{-1}(q_0)$ and $q_{m-1}$ are coprime.
Similarly, all zeros of $\mathtt{u}$ are outside
$\Omega_\omega^{(m)}$ and hence $\mathtt{u}$ and $q_{m-1}$ 
are coprime. By construction, $b$ has zeros inside 
$\Omega_\omega^{(m)}$, so all zeros of $\sigma^{-1}(b)$ must be
inside $\sigma^{-1}(\Omega_\omega^{(m)})$. At the same time,
the zeros of $q_{m-1}$ have real parts in $[\omega+2(m-1),\omega+2m)$
and hence are outside $\sigma^{-1}(\Omega_\omega^{(m)})$. Therefore
$q_{m-1}=1$. Applying $\mathtt{u}x^{-1}$ and using the arguments
from the previous paragraph gives $q_j=1$, for all $j$. This 
completes the proof.
\end{proof}

From Lemma~\ref{lem-s2.5-2}, we have $L_\beta=U(\mathfrak{g})v_0$.
Now we want to describe the $\mathbb{C}$-basis of this module
inside $N_\beta$ explicitly.

\begin{theorem}\label{thm-s2.5-7}
Let $\beta=g_{\boldsymbol{\zeta}}$ as above be such that 
$N_\beta$ is simple. Then $L_\beta$ is spanned, as a vector space
over $\mathbb{C}$, by the following linearly independent elements:
\begin{enumerate}[$($a$)$]
\item\label{thm-s2.5-7.0} $\mathtt{h}^iv_j$,
where $i\in\mathbb{Z}_{\geq 0}$ and $j\in\{0,1,\dots,m-1\}$.
\item\label{thm-s2.5-7.1} 
$\frac{1}{(\mathtt{h}-\lambda-2mi-2j)^{k}}\cdot v_j$, 
where $j\in\{0,1,\dots,m-1\}$, 
$i\in\mathbb{Z}_{\geq 0}$ and $1\leq k\leq -\mathtt{m}(\lambda)$,
for each $\lambda\in \Omega_\omega^{(m)}$ such that $\mathtt{m}(\lambda)<0$;
\item\label{thm-s2.5-7.2} 
$\frac{1}{(\mathtt{h}-\lambda+2mi+2(m-j))^{k}}\cdot v_j$, 
where $j\in\{0,1,\dots,m-1\}$,  $i\in\mathbb{Z}_{\geq 0}$ and 
\begin{displaymath}
1\leq k\leq \mathtt{m}(\lambda)-
|\{s\in\{1,2\}\,:\, r_s\in\{\lambda-2,\lambda-4,\dots,\lambda-2mi-2(m-j)\}\}|,
\end{displaymath}
for each $\lambda\in \Omega_\omega^{(m)}$ such that $\mathtt{m}(\lambda)>0$.
\end{enumerate}
\end{theorem}

\begin{proof}
Let $K_\beta$ be the subspace of $N_{\beta}$ given by the 
$\mathbb{C}$-basis in the formulation of the theorem. We need
to prove two things: that $K_\beta$ is closed under the action of 
$\mathfrak{g}$ and that any element of $K_\beta$ belongs to
$L_\beta$. Our strategy of the proof is to adjust 
the arguments in the proof of \cite[Theorem~9]{GKM} to our situation.
One can compare the formulation of our theorem with that of 
\cite[Theorem~9]{GKM} and see how the present formulation 
is adjusted taking into account the $\mathbf{C}_m$-grading that we have.

To start with, we note that $L_\beta$ contains 
$\mathbb{C}[\mathtt{h}]v_j$, for all $j$, as 
follows from Lemma~\ref{lem-s2.5-2}.
Now, we start with some $\lambda\in \Omega_\omega^{(m)}$ such that 
$\mathtt{m}(\lambda)<0$. Applying $x^m$ to $v_0$ gives
$\frac{a(\mathtt{h})}{b(\mathtt{h})}v_0$ (up to a non-zero scalar).
Note that $\frac{a(\mathtt{h})}{b(\mathtt{h})}$ has
$(\mathtt{h}-\lambda)^{\mathtt{m}(\lambda)}$ as a factor.
The numerator $a(\mathtt{h})$ is coprime with the denominator
and hence, writing $\frac{a(\mathtt{h})}{b(\mathtt{h})}$ in the basis
$\mathbf{B}$, the element $(\mathtt{h}-\lambda)^{\mathtt{m}(\lambda)}$
will appear with a non-zero coefficient. Applying
$\frac{b(\mathtt{h})}{(\mathtt{h}-\lambda)^k}$, for $k=1,2,\dots,-\mathtt{m}(\lambda)$,
and removing elements that belong to  
$\mathbb{C}[\mathtt{h}]v_0$, gives that $L_\beta$ contains
$\frac{1}{(\mathtt{h}-\lambda)^k}\cdot v_0$, 
for $k=1,2,\dots,-\mathtt{m}(\lambda)$.
Applying to these elements $x^m$ again and again and repeating the
arguments above gives that $L_\beta$ contains
$\frac{1}{(\mathtt{h}-\lambda-2mi)^k}\cdot v_0$, 
for $k=1,2,\dots,-\mathtt{m}(\lambda)$
and for all $i\in\mathbb{Z}_{\geq 0}$. 
Applying $x$ to 
$\frac{1}{(\mathtt{h}-\lambda-2mi)^k}\cdot v_0$ gives 
$\frac{1}{(\mathtt{h}-\lambda-2mi-2)^k}\cdot v_1$ and so on,
which justifies the elements from \eqref{thm-s2.5-7.1}.

Let $t_1,t_2\in\Omega_\omega^{(2)}$
be the unique elements for which there exist $n_1,n_2\in\mathbb{Z}_{>0}$
such that $t_1-r_1=4n_1$ and $t_2-r_2=4n_2$. Such $t_1$ and $t_2$ exist due to
our assumption on $\omega$. Now we want to justify the basis elements from
\eqref{thm-s2.5-7.2} applying powers of $\mathtt{f}$.
Here we point out a crucial difference with the previous case: the action of
$\mathtt{f}$ is given by 
$\frac{1}{4}\left(\vartheta-(\mathtt{h}+1)^2\right)x^{-1}$ and not
just by $x^{-1}$, so we need to take into account the additional
factor $\vartheta-(\mathtt{h}+1)^2$ whose contribution
leads to potential cancellations with some factors in the denominator.
For these cancellations to happen we need to have $\mathtt{m}(t_i)>0$,
for some $i\in \{1,2\}$. We also need to remember that the coefficient
appearing in the matrix of $x^{-1}$ is 
$\frac{1}{c}\cdot\frac{b(\mathtt{h}+2)}{a(\mathtt{h}+2)}$, 
see \eqref{eq-xandxinv2}.
Applying the arguments in the previous paragraph,
we see that the formulae in \eqref{thm-s2.5-7.2}
account exactly for such cancellations. This completes the proof of the 
inclusion  $K_\beta\subset L_\beta$.

To argue that $K_\beta$ is closed under the action of $\mathfrak{g}$,
we just need to check that, applying $\mathtt{e}$ and $\mathtt{f}$
to the elements that generate $K_\beta$, we cannot produce any new
denominators. Indeed, the only denominators that can appear come from
the zeros and the poles of $\beta$, as well as from their 
$\sigma^{\pm 1}$-shifts. Additionally, we need to account for 
potential cancellations between the factors of 
$\vartheta-(\mathtt{h}+1)^2$ and $\sigma^{i}(a(\mathtt{h}))$.
These are taken care of by the formulae in 
\eqref{thm-s2.5-7.2}. This completes the proof of the theorem.
\end{proof}

\begin{corollary}\label{cor-fing}
Let $\beta=g_{\boldsymbol{\zeta}}$ as above be such that 
$N_\beta$ is simple. Then $L_\beta$ is finitely generated over 
$\mathbb{C}[\mathtt{h}]$ if and only if, for any 
$\lambda\in\Omega_\omega^{(m)}$, we have 
\begin{displaymath}
0\leq \mathtt{m}(\lambda)\leq |\{j\in\{1,2\}\,:\, r_j\in \lambda-2\mathbb{Z}_{>0}\}|. 
\end{displaymath}
\end{corollary}

\begin{proof}
The module $L_\beta$ is finitely generated over $\mathbb{C}[\mathtt{h}]$
if and only if the basis of $L_\beta$ described in Theorem~\ref{thm-s2.5-7}
is such that the total number of basis elements listed in \eqref{thm-s2.5-7.1}
and \eqref{thm-s2.5-7.2} is finite. From the formulation of 
Theorem~\ref{thm-s2.5-7}, we directly see that this is equivalent to the
condition of the corollary.
\end{proof}

We remark that being torsion-free and, at the same time, finitely generated 
over $\mathbb{C}[\mathtt{h}]$ is  equivalent to being free of finite rank.
Therefore Corollary~\ref{cor-fing} provides a classification of all 
$\mathbf{C}_m$-gradable simple $\mathfrak{sl}_2$-module that 
are free of rank $m$ over $\mathbb{C}[\mathtt{h}]$.

\begin{example}\label{ex-sl2}
{\em
Let $m=2$ and $\vartheta=0$. Then $\vartheta-(\mathtt{h}+1)^2=-(\mathtt{h}+1)^2$ 
has root $-1$ of multiplicity $2$. Take $\omega=0$ and let
$\beta=\mathtt{h}-2$ (so $c=1$ and $\mathtt{m}$ has support $2$
where it has value $1$). Clearly, $\beta\neq \alpha\sigma(\alpha)$, 
for any $\alpha$ and hence $x^2-\beta$ is left irreducible. Note
that $-1\not\in 2+2\mathbb{Z}$. Therefore $L_\beta$ in this case 
has a basis consisting of all $\mathtt{h}^iv_0$ and all
$\mathtt{h}^iv_1$, where $i\in\mathbb{Z}_{\geq 0}$, as well as 
all  $\frac{1}{\mathtt{h}+2+4i}v_0$
and all $\frac{1}{\mathtt{h}+4i}v_1$, where $i\in\mathbb{Z}_{\geq 0}$.
} 
\end{example}

\begin{example}\label{ex-sl2-2}
{\em
Let $m=2$ and $\vartheta=0$. Then $\vartheta-(\mathtt{h}+1)^2=-(\mathtt{h}+1)^2$ 
has root $-1$ of multiplicity $2$. Take $\omega=0$ and let
$\beta=\mathtt{h}-3$ (so $c=1$ and $\mathtt{m}$ has support $3$
where it has value $1$). Clearly, $\beta\neq \alpha\sigma(\alpha)$, 
for any $\alpha$ and hence $x^2-\beta$ is left irreducible. Note
that $-1= 3-2\cdot 2$. Therefore $L_\beta$ in this case 
has a basis consisting of all $\mathtt{h}^iv_0$ and all
$\mathtt{h}^iv_1$, where $i\in\mathbb{Z}_{\geq 0}$, as well as 
the element $\frac{1}{\mathtt{h}-1}v_1$.
We see that, in this case, $L_\beta$ is free over $\mathbb{C}[\mathtt{h}]$.
} 
\end{example}

\section{Applications to simple modules over the first Weyl algebra}\label{s3}

\subsection{General setup}\label{s3.1}

Recall that the {\em first Weyl algebra} $\mathbf{A}_1$
is defined by the following presentation: 
$\mathbf{A}_1:=\langle a,b\,:\, ab-ba=1\rangle$.
The algebra $\mathbf{A}_1$ is simple. The algebra 
$\mathbf{A}_1$ admits a natural $\mathbb{Z}$-grading in which
the degree of $a$ is $1$ and the degree of $b$ is $-1$.
For any positive integer $m$, this induces a grading
by $\mathbf{C}_m$.

Just like in Section~\ref{s1}, we consider $\Bbbk:=\mathbb{C}(\mathtt{h})$,
its automorphism $\sigma$ of $\Bbbk$ and the corresponding skew 
Laurent polynomial algebra  $R:=\Bbbk[x,x^{-1},\sigma]$.
Mapping $a$ to $x$ and $b$ to $-\frac{\mathtt{h}}{2}x^{-1}$
defines an injective algebra homomorphism $\psi$ from $\mathbf{A}_1$ into $R$.
This allows us to view $R$-modules as $\mathbf{A}_1$-modules
by pulling back via $\psi$.

\subsection{Simple gradable torsion free $\mathbf{A}_1$-modules}\label{s3.2}

Let $m$ be a positive integer. Take $\omega=1$. Then, by 
Theorem~\ref{thm-s1.8-9}, the isomorphism classes of 
simple $R$-modules of the form $R/R(x^m-\beta)$ are in bijection
with the regular orbits of $\mathbf{C}_m$ on $\Psi^{(m)}$.
Let $\boldsymbol{\zeta}\in \Psi^{(m)}$, be such that the
corresponding $R$-module $R/R(x^m-g_{\boldsymbol{\zeta}})$ is simple.
Denote by $F_{\boldsymbol{\zeta}}$ the simple socle of the 
restriction of $R/R(x^m-g_{\boldsymbol{\zeta}})$ to $\mathbf{A}_1$.
Note that $F_{\boldsymbol{\zeta}}\cong F_{\boldsymbol{\zeta}'}$
if and only if $\boldsymbol{\zeta}\in \mathbf{C}_m\cdot \boldsymbol{\zeta}'$.
Moreover, the modules of the form $F_{\boldsymbol{\zeta}}$
exhaust $\mathbf{C}_m$-gradable simple $\mathbf{A}_1$-modules
which are torsion free of rank $m$ over the subalgebra $\mathbb{C}[ab]$.

In analogy with Theorem~\ref{thm-s2.5-7},
the module $F_{\boldsymbol{\zeta}}$ can be described explicitly 
as the submodule of $R/R(x^m-g_{\boldsymbol{\zeta}})$ with the following
$\mathbb{C}$-basis:
\begin{itemize}
\item $\mathtt{h}^iv_j$,
where $i\in\mathbb{Z}_{\geq 0}$ and $j\in\{0,1,\dots,m-1\}$;
\item
$\frac{1}{(\mathtt{h}-\lambda-2mi-2j)^{k}}\cdot v_j$, 
where $j\in\{0,1,\dots,m-1\}$, 
$i\in\mathbb{Z}_{\geq 0}$ and $1\leq k\leq -\mathtt{m}(\lambda)$,
for each $\lambda\in \Omega_1^{(m)}$ such that $\mathtt{m}(\lambda)<0$;
\item 
$\frac{1}{(\mathtt{h}-\lambda+2mi+2(m-j))^{k}}\cdot v_j$, 
where  $j\in\{0,1,\dots,m-1\}$, 
$i\in\mathbb{Z}_{>0}$ and also we have $1\leq k\leq \mathtt{m}(\lambda)$,
for each $\lambda\in \Omega_1^{(m)}\setminus 2\mathbb{Z}$ 
such that $\mathtt{m}(\lambda)>0$;
\item 
$\frac{1}{(\mathtt{h}-\lambda+2mi+2(m-j))^{k}}\cdot v_j$, 
where  $j\in\{0,1,\dots,m-1\}$,
$i\in\mathbb{Z}_{>0}$ and also we have $1\leq k\leq \mathtt{m}(\lambda)-1$,
for each $\lambda\in \big(\Omega_1^{(m)}\cap 2\mathbb{Z}\big)$ 
such that $\mathtt{m}(\lambda)>0$;
\item 
$\frac{1}{(\mathtt{h}-\lambda+2(m-j))^{k}}\cdot v_j$, 
where  $j\in\{0,1,\dots,m-1\}$ and $1\leq k\leq \mathtt{m}(\lambda)-1$,
for each $\lambda\in \{2,4,\dots,2(m-j)\}$ 
such that $\mathtt{m}(\lambda)>0$;
\item 
$\frac{1}{(\mathtt{h}-\lambda+2(m-j))^{k}}\cdot v_j$, 
where  $j\in\{0,1,\dots,m-1\}$ and $1\leq k\leq \mathtt{m}(\lambda)$,
for each $\lambda\in \big(\Omega_1^{(m)}\setminus \{2,4,\dots,2(m-j)\}\big)$ 
such that $\mathtt{m}(\lambda)>0$.
\end{itemize}

Note that $F_{\boldsymbol{\zeta}}$ is finitely generated over 
$\mathbb{C}[\mathtt{h}]$ if and only if $\mathtt{m}$ is either 
the zero function or $\mathtt{m}(\lambda)\neq 0$ implies 
both $\lambda\in 2\mathbb{Z}$  and $\mathtt{m}(\lambda)=1$.

\begin{example}\label{ex-a1-1}
{\em
Let $m=3$ and $\mathtt{m}$ be given by $\mathtt{m}(1)=-\mathtt{m}(3)=\mathtt{m}(4)=1$
with all other values being $0$. Then $F_{\boldsymbol{\zeta}}$ has a basis
consisting of the following elements: $\mathtt{h}^iv_j$, where 
$i\in\mathbb{Z}_{\geq 0}$ and $j=0,1,2$, as well as the elements
$\frac{1}{\mathtt{h}-3-6i-2j}\cdot v_j$, where 
$i\in\mathbb{Z}_{\geq 0}$ and $j=0,1,2$; the elements
$\frac{1}{\mathtt{h}-1+6i+2(3-j)}\cdot v_j$, where 
$i\in\mathbb{Z}_{\geq 0}$ and $j=0,1,2$; and the element
$\frac{1}{\mathtt{h}-2}\cdot v_2$.
} 
\end{example}

\vspace{2mm}

\noindent
(V.~M.) Department of Mathematics, Uppsala University, Uppsala, SWEDEN \\
Email address: {\tt mazor\symbol{64}math.uu.se}

\noindent
(S.~X.) School of Mathematics, Hohai University,  Nanjing, Jiangsu, P.~R.~CHINA \\
Email address: {\tt shuoyang\symbol{64}hhu.edu.cn}

\end{document}